\documentclass[12pt]{article}
\usepackage{amscd,amssymb,latexsym,amsmath,amsthm,txfonts}
\usepackage{geometry}
\usepackage{graphicx}
\usepackage{graphicx, float,subcaption}
\usepackage{cite}
\usepackage{hyperref}
\usepackage{breqn}
\usepackage{adjustbox}
\usepackage{multicol}
\newtheorem{The}{Theorem}[section]
\newtheorem{Def}{Definition}[section]

\numberwithin{equation}{section}
\begin{document}
{\title{Solving systems of multi-term fractional PDEs: Invariant subspace approach. }
	\author{Sangita Choudhary and Varsha Daftardar-Gejji\footnote{Author for correspondence}\\
			Department of Mathematics\\
			Savitribai Phule Pune University, Pune 411007.\\
			\textit{schoudhary1695@gmail.com,\, vsgejji@gmail.com,}\\
			\textit{vsgejji@unipune.ac.in}
			}
		\date{}
		\maketitle }
	\begin{abstract}
In the present paper invariant subspace method has been extended for solving systems of multi-term fractional partial differential equations (FPDEs) involving both time and space fractional derivatives. Further the method has also been employed for solving multi-term fractional PDEs in $(1+n)$ dimensions. A diverse set of examples is solved to illustrate the method. 
		
		 \medskip
		 
	\end{abstract}
	\section{Introduction}
	\paragraph{}Fractional order differential equations (FODEs) are receiving increasing attention owing to their applicability to almost all branches of science and engineering. It has been established that fractional order partial differential equations (FPDEs) provide appropriate framework for description of anomalous and non-Brownian diffusion. They are more effective while formulating processes having memory effects as fractional derivatives are non-local in nature \cite{bhalekar2011fractional,debnath2003recent,podlubny1998fractional}.
	
	Hence solving FODEs, FPDEs, especially nonlinear ones, is a challenging task and currently an active area of research. In pursuance to this researchers have developed new numerical/ analytical methods for solving FPDEs such as Adomian decomposition method (ADM)\cite{adomian1989solution}, new iterative method (NIM)\cite{bhalekar2008new},  iterative Laplace transform method \cite{jafari2013new}, method of separation of variables, homogeneous balanced principle \cite{wu2018method}, Lie symmetry analysis method \cite{ovsiannikov2014group} and so on.
	
	One of the analytical methods for solving PDEs is invariant subspace method developed by Galaktionov and Svirshchevskii \cite{galaktionov2006exact}. Invariant subspace method was employed for solving time fractional PDEs by many authors \cite{gazizov2013construction, sahadevan2016exact}. Further Choudhary and Daftardar-Gejji \cite{choudhary2017invariant} extended the method for FPDEs having both time and space fractional derivatives. In 2009 a classification of two-component nonlinear diffusion equations based on invariant subspace method was proposed \cite{qu2009classification}. Sahadevan \textit{et al.} did extensive study of Lie symmetry analysis and invariant subspace method for deriving exact solutions of the coupled FPDEs with fractional time derivative \cite{sahadevan2017lie}.
	
	In the present paper we develop invariant subspace method for finding analytic solutions of systems of multi-term FPDEs having both time and space  fractional derivatives. Further the method is employed for solving FPDEs in (1+n) dimensions. In the proposed method system of FPDEs and FPDEs in higher dimensions are reduced to respective system of FODEs which can be solved by known methods. Invariant subspace method is also used to solve FPDEs with fractional differential operator involving mixed fractional partial derivatives.
	
	The organization of the paper is as follows. Section 2, deals with preliminaries and notations. In Section 3, we develop theory of invariant subspace method for r-coupled FPDEs, which is followed by illustrative examples. In Section 4 we extend invariant subspace method for FPDEs in (1+n) dimensions and explain the method with a variety of illustrative examples. Concluding remarks are made in Section 5.

	\section{Preliminaries and Notations}
	In this section, we introduce notations, definitions and preliminaries which are used in the present article. For more details readers may refer to \cite{diethelm2010analysis,miller1993introduction,podlubny1998fractional}.
	\begin{Def}The Riemann-Liouville (R-L) fractional integral of order $\alpha > 0$ of function $f$ is defined as
		\begin{equation*}
		I^{\alpha}f(t)=\frac{1}{\Gamma(\alpha)}\int\limits_{0}^{t}\frac{f(\tau)}{(t-\tau)^{1-\alpha}}d\tau,~~0<t\leq b.
		\end{equation*}
	\end{Def}
	\begin{Def}
		Caputo fractional derivative of order $\alpha>0$ of $f$ is defined as
		\begin{align*}
		\frac{ d^{\alpha}f(t)}{ dt^\alpha}=
		\left\{\begin{array}{lcr}I^{n-\alpha}D^{n}f(t)=\dfrac{1}{\Gamma(n-\alpha)}\displaystyle\int\limits_{0}^{t}\frac{f^{(n)}(\tau)}{(t-\tau)^{\alpha-n+1}}d\tau,~~~~~\mbox{$n-1 < \alpha < n,$}\\
		f^{(n)}(t),\hspace{6.4cm} \mbox{$\alpha=n,~n \in \mathbb{N}.$}
		\end{array}\right.
		\end{align*}
	\end{Def}
		\begin{Def}
		Riemann-Liouville (R-L) fractional derivative of order $\alpha>0$ of $f$ is defined as
			\begin{align*}
			\frac{^{RL} d^{\alpha}f(t)}{ dt^\alpha}=
			\left\{\begin{array}{lcr}D^{n}I^{n-\alpha}f(t)=\dfrac{ d^n}{ dt^n}\left(\dfrac{1}{\Gamma(n-\alpha)}\displaystyle\int\limits_{0}^{t}\frac{f(\tau)}{(t-\tau)^{\alpha-n+1}}dx\right),~~~~~\mbox{$n-1 < \alpha < n,$}\\
			f^{(n)}(t),\hspace{7.4cm} \mbox{$\alpha=n,~n \in \mathbb{N}.$}
			\end{array}\right.
			\end{align*}
		\end{Def}
	\noindent	R-L integral, Caputo derivative and R-L derivative satisfy the following properties for $\lceil\alpha\rceil=n, n\in \mathbb{N}$ \cite{diethelm2010analysis}:
	\begin{enumerate}
\item $	I^\alpha t^\gamma=\frac{\Gamma(\gamma+1)}{\Gamma(\gamma+\alpha+1)}t^{\gamma+\alpha},\qquad\qquad\qquad~~~ \text{if}~ \gamma>-1,~t>0.$
	
\item $
\dfrac{ d^{\alpha}t^\gamma}{ dt^\alpha} =\left\{\begin{array}{lcr}
0,\qquad\qquad\qquad\qquad\quad~~\text{if}~\gamma\in \{0,1,2,\dots,n-1\},\vspace{.1cm}\\
\frac{\Gamma(\gamma+1)}{\Gamma(\gamma-\alpha+1)}t^{\gamma-\alpha},\qquad\qquad\quad~~\text{if}~\gamma\in \mathbb{N}~\text{and}~\gamma \geq n, ~\text{or}~\gamma\notin \mathbb{N}~\text{and}~\gamma>n-1.
\end{array}\right.$

\item $	\dfrac{^{RL} d^{\alpha}t^\gamma}{~~ dt^\alpha} =\left\{\begin{array}{lcr}
0,\qquad\qquad\qquad\qquad~~~ \text{if}~\gamma>-1 ~\text{and} ~\alpha-\gamma\in \{0,1,\dots,n-1\},\vspace{.1cm}\\
\frac{\Gamma(\gamma+1)}{\Gamma(\gamma-\alpha+1)}t^{\gamma-\alpha}, \qquad\qquad~~~\text{if} ~\gamma>-1~\text{and}~\alpha -\gamma\notin \mathbb{N}.
\end{array}\right.$

\item $I^\alpha \left(\dfrac{ d^{\alpha} f(t)}{ dt^\alpha}\right)=f(t)-\displaystyle\sum_{k=0}^{n-1}D^{(k)}f(0)\frac{t^k}{k!},~~ n-1<\alpha<n,~t>0.$
	\end{enumerate}
\textbf{Note:}	In the property (2) condition $\gamma>n-1$ 	is very crucial as $$\dfrac{ d^{\alpha}(t^{-\alpha})}{ dt^\alpha},$$
 is not defined in case of Caputo derivate for $0< \alpha<1$. In the literature many authors  are mistakenly ignoring the underlying required condition $\gamma>n-1$ (here $\gamma=-\alpha>0$ is required but $-1<-\alpha<0$). Therefore
  $$\dfrac{ d^{\alpha}(t^{-\alpha})}{ dt^\alpha}=\dfrac{\Gamma(1-\alpha)}{\Gamma(1-2\alpha)}t^{-2\alpha},$$ is not valid in case of Caputo derivative, though it holds correct for R-L derivative:
 $$\dfrac{^{RL} d^{\alpha}(t^{-\alpha})}{ dt^\alpha}=\dfrac{\Gamma(1-\alpha)}{\Gamma(1-2\alpha)}t^{-2\alpha},~\gamma=-\alpha>-1.$$
 
In the present work we denote fractional partial derivative $\dfrac{\partial ^{k\gamma}}{\partial t^{k\gamma}}$ and $ \dfrac{^{RL}\partial ^{k\gamma}}{\partial t^{k\gamma}}$ (Caputo and RL partial derivative respectively) as sequential fractional partial derivative \cite{miller1993introduction}, \textit{viz}.,
\begin{equation*}
\dfrac{\partial ^{k\gamma}f}{\partial t^{k\gamma}}=\underbrace{\dfrac{\partial ^{\gamma}}{\partial t^{\gamma}}\dfrac{\partial ^{\gamma}}{\partial t^{\gamma}}\cdots\dfrac{\partial ^{\gamma}f}{\partial t^{\gamma}}}_{k-times},~~~\dfrac{^{RL}\partial ^{k\gamma}f}{\partial t^{k\gamma}}=\underbrace{\dfrac{^{RL}\partial ^{\gamma}}{\partial t^{\gamma}}\dfrac{^{RL}\partial ^{\gamma}}{\partial t^{\gamma}}\cdots\dfrac{^{RL}\partial ^{\gamma}f}{\partial t^{\gamma}}}_{k-times}.
\end{equation*}

	\begin{Def}Two-parametric Mittag-Leffler function is defined as	\begin{equation*}E_{\alpha, \beta} (z)=\sum\limits_{k=0}^{\infty}\frac{z^{k}}{\Gamma(k\alpha+\beta)},~~\alpha>0,~\beta>0.
	\end{equation*}\end{Def}
\noindent The $n-$th order derivative of $E_{\alpha,\beta}(z)$ is given by
	\begin{equation*} E_{\alpha,\beta}^{(n)}(z)=\frac{ d^{n}}{ dz^{n}}E_{\alpha,\beta}(z)=\displaystyle\sum\limits_{k=0}^{\infty}\frac{(k+n)! z^{k}}{k! \Gamma(\alpha k+\alpha n+\beta)},~~n=0,1,2,\ldots.\end{equation*}
The $\alpha-$th order Caputo derivative of $E_\alpha(at^\alpha)$ is
	\begin{equation*}
\frac{	 d^\alpha}{ dt^\alpha} [E_\alpha(at^\alpha)]=aE_\alpha(at^\alpha),~~\alpha>0,~a\in \mathbb{R}.
	\end{equation*}
Generalized fractional trignometric functions for $\lceil\gamma\rceil=n$ are defined as \cite{bonilla2007fractional}
\begin{align*}
\cos_{\gamma}(\lambda t^{\gamma})&=\mathcal{R}e[E_{\gamma}(i\lambda ^{\gamma})]=\sum_{k=0}^{\infty}\frac{(-1)^k\lambda^{2k}t^{(2k)\gamma}}{\Gamma(2k\gamma+1)},\\
\sin_{\gamma}(\lambda t^{\gamma})&=\mathcal{I}m[E_{\gamma}(i\lambda ^{\gamma})]=\sum_{k=0}^{\infty}\frac{(-1)^k\lambda^{2k+1}t^{(2k+1)\gamma}}{\Gamma((2k+1)\gamma+1)}.
\end{align*}
The fractional trigonometric functions satisfy the following properties
\begin{equation*}
\frac{	 d^\alpha}{ dt^\alpha}[\cos_{\gamma}(\lambda t^{\gamma})]=-\lambda\sin_{\gamma}(\lambda t^{\gamma}),~\frac{	 d^\alpha}{ dt^\alpha}[\sin_{\gamma}(\lambda t^{\gamma})]=\lambda\cos_{\gamma}(\lambda t^{\gamma}).
\end{equation*}	
		\noindent Laplace transform of the Caputo derivative of order $\alpha$ is,
	$$\mathcal{L}\left\{\frac{	 d^\alpha f(t)}{ dt^\alpha};s\right\}=s^{\alpha}\hat{f}(s)-\sum\limits_{k=0}^{n-1}s^{\alpha-k-1}f^{(k)}(0),~~n-1<\alpha\leq{n},~n\in\mathbb{N},~\mathcal{R}e(s)>0,$$where $\hat{f}(s)=\mathcal{L}\{f(t);s\}=\int\limits_{0}^{\infty}e^{-st}f(t)dt,~~s\in\mathbb{R}.$\\
	Laplace transform of $\varepsilon_n(t,a;\alpha,\beta):=t^{\alpha n+\beta-1}E_{\alpha,\beta}^{(n)}(\pm at^{\alpha})$ has the form 
	\begin{equation}\label{2.P1}
	\mathcal{L}\{\varepsilon_n(t,a;\alpha,\beta);s\}=\frac{n!s^{\alpha-\beta}}{(s^{\alpha}\mp a)^{n+1}},~~\mathcal{R}e(s)>|a|^{\frac{1}{\alpha}},~~n=0,1,2,\ldots.
	\end{equation}

	\section{System of FPDES}
In this section we extend invariant subspace method for solving systems of FPDEs. We introduce the following notations:\\
	Let $f=(f_1, f_2,\dots, f_r) = (f_1(t,x), f_2(t,x),\dots, f_r(t,x))\in \mathbb {R}^r,$ where $ t>0, x \in \mathbb{R}.$	\begin{align*}
	N^1[f]&:= (N_1^1[f],N_2^1[f],\dots,N_r^1[f])\in \mathbb {R}^r, \textrm{where}\\
N_p^1[f]&:=\hat{N}_p^1\left [x, f_1,f_2,\dots f_r,\frac{\partial^{\beta}f_1}{\partial x^{\beta}},\dots, \frac{\partial^{\beta}f_r}{\partial x^{\beta}},\dots, \frac{\partial^{k\beta}f_1}{\partial x^{k\beta}},\dots, \frac{\partial^{k\beta}f_r}{\partial x^{k\beta}}\right ], 1\leq p\leq r,~\text{and}\\
N^2[f]&= (N_1^2[f],N_2^2[f],\dots,N_r^2[f])\in \mathbb {R}^r, \textrm{where}\\
	N_p^2[f]&=\hat{N}_p^2\left [x, f_1,\dots f_r,\frac{\partial^{\beta}f_1}{\partial x^{\beta}},\dots, \frac{\partial^{\beta}f_r}{\partial x^{\beta}},\dots, \frac{\partial^{\beta+k-1}f_1}{\partial x^{\beta+k-1}},\dots, \frac{\partial^{\beta+k-1}f_r}{\partial x^{\beta+k-1}}\right ], 1\leq p\leq r,~k\in \mathbb{N},
	\end{align*}
are linear/ non-linear fractional differential operators. Let $F = (F_1, F_2,\dots, F_r)\in \mathbb{R}^r$ be such that
\begin{equation*}
F_p=\sum_{i=1}^{m_p}\lambda _{pi}\frac{\partial ^{\gamma(i,p)}f_p(t,x)}{\partial t ^{\gamma(i,p)}},~ p=1,\dots, r,
\end{equation*}
where $\gamma(i,p)=i\alpha_p~ \text{or}~ \gamma(i,p)=\alpha_p+i-1.$\\
In this article $\frac{\partial^{j\beta}(\cdot)}{\partial x^{j\beta}}$ and $\frac{\partial^{\beta+j-1}(\cdot)}{\partial x^{\beta+j-1}},j=1,\dots,k$ denote Caputo derivatives with respect to variable $x$ and $\frac{\partial ^{\gamma(i,p)}f_p(\cdot)}{\partial t ^{\gamma(i,p)}}$ denotes Caputo or Riemann-Liouville  derivative with respect to variable $t$. $\lceil \alpha_p \rceil=s_p, \lceil \beta \rceil=s,$ where $s_p, s \in \mathbb{N}$. Henceforth throughout the article $p=1,\dots,r.$\\

We consider the system of coupled FPDEs as
\begin{equation}\label{2.SFPDE}
F=N^l[f],~l=1,2.
\end{equation}Eq. (\ref{2.SFPDE}) consists of 4-kinds of different systems for p=1, 2,\dots, r, \textit{viz.,}
\begin{multicols}{2}
\begin{itemize}
	\item $ F_p=\sum_{i=1}^{m_p}\lambda _{pi}\frac{\partial ^{i\alpha_p}f_p(t,x)}{\partial t ^{i\alpha_p}}=N_p^1[f],$
\item $ F_p=\sum_{i=1}^{m_p}\lambda _{pi}\frac{\partial ^{\alpha_p+i-1 }f_p(t,x)}{\partial t ^{\alpha_p+i-1}}=N_p^1[f],$
\end{itemize}
	\columnbreak
\begin{itemize}
	\item $ F_p=\sum_{i=1}^{m_p}\lambda _{pi}\frac{\partial ^{i\alpha_p}f_p(t,x)}{\partial t ^{i\alpha_p}}=N_p^2[f],$
	\item $ F_p=\sum_{i=1}^{m_p}\lambda _{pi}\frac{\partial ^{\alpha_p+i-1 }f_p(t,x)}{\partial t ^{\alpha_p+i-1}}=N_p^2[f].$
\end{itemize}
\end{multicols}
\noindent Further note that $F$ is a function of either Caputo derivative or Riemann-Liouville derivative, whereas $N^l[f]$ always involves only Caputo fractional derivatives.
\subsection{Invariant subspace method for systems of FPDEs}
Let $I= I_1^{n_1}\times I_2^{n_2}\times \cdots \times I_r^{n_r}$ represent a linear space where $I_p^{n_p}$ denotes an $n_p$ dimensional linear subspace over $\mathbb{R}$ spanned by $n_p$ linearly independent functions $\{\phi_p^j(x)\}_{j=1}^{n_p}$, \textit{i.e.,}
\begin{equation*}
I_p^{n_p}=\mathfrak{L}\left\{\phi_p^1(x),\phi_p^2(x),\dots, \phi_p^{n_p}(x)\right\}=\left\{\sum_{j=1}^{n_p}k_{pj}\phi_p^j(x)~\Big|~k_{pj}\in \mathbb{R},~ j=1,2,\dots, n_p\right\},~ \forall p.
\end{equation*}
$I$ is said to be invariant with respect to vector differential operators $N^l, l=1,2$ if $N^l$ satisfies the following condition $\forall p$.
\begin{align*}
N_p^l&:I_1^{n_1}\times I_2^{n_2}\times \cdots \times I_r^{n_r} \longrightarrow I_p^{n_p},~ l=1,2.
\end{align*}
Thus there exist expansion coefficients $\psi_p^j (j=1,2,\dots, n_p)$ of $N_p^l[f]$ with respect to the basis functions $\{\phi_p^j(x)\}_{j=1}^{n_p}$ such that 
\begin{align*}
N_p^l\left[\sum_{j=1}^{n_1}k_{1j}\phi_1^j(x),\dots, \sum_{j=1}^{n_r}k_{rj}\phi_r^j(x)\right]=\sum_{j=1}^{n_p}\psi_p^j(k_{11},k_{12},\dots, k_{1n_1},\dots,k_{r1},\dots, k_{rn_r})\phi_p^j(x),\\ (k_{p1},k_{p2}\dots, k_{pn_p})\in \mathbb{R}^{n_p}, \forall p,~l=1,2.
\end{align*}
\begin{The}\label{2.THE1}
	If a finite dimensional linear subspace $I= I_1^{n_1}\times I_2^{n_2}\times \cdots \times I_r^{n_r}$ is invariant under the fractional differential operators $N^l[F], l=1, 2,$ then the system of FPDEs (\ref{2.SFPDE}) has a solution of the form
	\begin{equation}\label{2.THE1.1}
	f_p(t,x)=\sum_{j=1}^{n_p}K_{pj}(t)\phi _p^j(x),~p=1, 2,\dots, r,
	\end{equation}
	where the coefficients $K_{pj}(t)$ satisfy the following system of FODEs
	\begin{align}\label{2.THE1.2}
	\sum_{i=1}^{m_p}\lambda_{pi}\frac{ d ^{\gamma(i,p)}K_{pj       }(t)}{ dt ^{\gamma(i,p)}}=\psi_p^j(K_{11}(t),\dots, K_{1n_1}(t),\dots,K_{r1}(t),\dots, K_{rn_r}(t)),~j=1,\dots, n_p,
	\end{align}
	where $\gamma=\gamma(i,p)=i\alpha_p~ \text{or}~ \gamma=\gamma(i,p)=\alpha_p+i-1,~ \forall p.$
\end{The}
\begin{proof}
	Let $f_p(t,x)=\displaystyle\sum_{j=1}^{n_p}K_{pj}(t)\phi _p^j(x),~\forall p.$\\
	Using the linearity property of fractional derivative we get
	\begin{align}\label{2.LHS1.1}
	F_p=\sum_{i=1}^{m_p}\lambda_{pi}\frac{\partial ^{\gamma(i,p)}f_p(t,x)}{\partial t ^{\gamma(i,p)}}=\sum_{i=1}^{m_p}\lambda_{pi}\frac{\partial ^{\gamma(i,p)}}{\partial t ^{\gamma(i,p)}}\left[\sum_{j=1}^{n_p}K_{pj}(t)\phi _p^j(x)\right]=\sum_{j=1}^{n_p}\left[\sum_{i=1}^{m_p}\lambda_{pi}\frac{ d ^{\gamma(i,p)}}{dt ^{\gamma(i,p)}}K_{pj}(t)\right]\phi _p^j(x),\nonumber\\\gamma=\gamma(i,p)=i\alpha_p~ \text{or}~ \gamma=\gamma(i,p)=\alpha_p+i-1,~ \forall p.
	\end{align}
	Given that the fractional differential operators $N^l[F], l=1, 2$ admit invariant subspace $I$, there exist basis functions $\phi_p^1(x), \phi_p^2(x),\dots, \phi_p^{n_p}(x)$ such that 
	\begin{align}\label{2.RHS1.1}
N_p^l\left[\sum_{j=1}^{n_1}k_{1j}\phi_1^j(x),\dots, \sum_{j=1}^{n_r}k_{rj}\phi_r^j(x)\right]=\sum_{j=1}^{n_p}\psi_p^j(k_{11},k_{12},\dots, k_{1n_1},\dots,k_{r1},\dots, k_{rn_r})\phi_p^j(x),\nonumber\\ (k_{p1},k_{p2}\dots, k_{pn_p})\in \mathbb{R}^{n_p}, \forall p,~l=1,2,
	\end{align}
	where $\{\psi^j_p\}'s$ are expansion coefficients of $N_p[I] \in I_p^{n_p}$ corresponding to $\{\phi_p^j\}'s$. Hence in view of Eq. (\ref{2.THE1.1}) and Eq. (\ref{2.RHS1.1}), we deduce
	\begin{align}\label{2.RHS1.2}
	N_p^l[f]&=N_p^l[f_1, f_2,\dots, f_r]=N_p^l\left[\sum_{j=1}^{n_1}K_{1j}(t)\phi_1^j(x),\sum_{j=1}^{n_2}K_{2j}(t)\phi_2^j(x),\dots, \sum_{j=1}^{n_r}K_{rj}(t)\phi_r^j(x)\right]\nonumber\\
	&=\sum_{j=1}^{n_p}\psi_p^j\left(K_{11}(t),\dots, K_{1n_1}(t),\dots,K_{r1}(t),\dots, K_{rn_r}(t)\right)\phi_p^j(x),~\forall p,~l=1,2.
	\end{align}
	Substituting (\ref{2.LHS1.1}) and (\ref{2.RHS1.2}) in (\ref{2.SFPDE}) we get
	\begin{align}\label{2.LHS=RHS1}
	\sum_{j=1}^{n_p}\left[	\sum_{i=1}^{m_p}\lambda_{pi}\frac{ d ^{\gamma(i,p)}K_{pj}(t)}{ dt ^{\gamma(i,p)}}-\psi_p^j(K_{11}(t),\dots, K_{1n_1}(t),\dots,K_{r1}(t),\dots, K_{rn_r}(t))\right]\phi_p^j(x)=0,~\forall p.
	\end{align}
	From the Eq. (\ref{2.LHS=RHS1}) and using the fact that $\{\phi_p^j\}'s$ are linearly independent, the system of FPDEs (\ref{2.SFPDE}) is reduced to the required system of FODEs (\ref{2.THE1.2}).
\end{proof}
\subsection{Illustrative examples}
\subsubsection{System of fractional version of generalized Burger's equations}Consider the following coupled generalized nonlinear fractional Burger's equations for $t>0,~\alpha \in  (0,1)\backslash \{1/2\}$ and $ \beta \in (0,1].$
	\begin{align}\label{2.1EX1}
	\frac{^{RL}\partial^{\alpha}f}{\partial t^{\alpha}}+a_0\frac{\partial^{2\beta}f}{\partial x^{2\beta}}+a_1f\left(\frac{\partial^{\beta}f}{\partial x^{\beta}}\right)+a_2\left(f\frac{\partial^{\beta}g}{\partial x^{\beta}}+g\frac{\partial^{\beta}f}{\partial x^{\beta}}\right)&=0,\nonumber\\
	\frac{^{RL}\partial^{\alpha}g}{\partial t^{\alpha}}+b_0\frac{\partial^{2\beta}g}{\partial x^{2\beta}}+b_1g\left(\frac{\partial^{\beta}g}{\partial x^{\beta}}\right)+b_2\left(f\frac{\partial^{\beta}g}{\partial x^{\beta}}+g\frac{\partial^{\beta}f}{\partial x^{\beta}}\right)&=0,
	\end{align}
	where  $a_0,a_1,a_2,b_0,b_1$ and $b_2(\neq0)$
	are arbitrary constants depending upon the system parameters such as Peclet number, Brownian diffusivity and Stokes velocity of
	particles due to gravity. \\
\noindent Comparing with the system of FPDE (\ref{2.SFPDE}) we conclude that 
\begin{align*}
N_1[f,g]&=-a_0\frac{\partial^{2\beta}f}{\partial x^{2\beta}}-a_1f\left(\frac{\partial^{\beta}f}{\partial x^{\beta}}\right)-a_2\left(f\frac{\partial^{\beta}g}{\partial x^{\beta}}+g\frac{\partial^{\beta}f}{\partial x^{\beta}}\right),\\
N_2[f,g]&=-b_0\frac{\partial^{2\beta}g}{\partial x^{2\beta}}-b_1g\left(\frac{\partial^{\beta}g}{\partial x^{\beta}}\right)-b_2\left(f\frac{\partial^{\beta}g}{\partial x^{\beta}}+g\frac{\partial^{\beta}f}{\partial x^{\beta}}\right),
\end{align*} are the corresponding nonlinear fractional differential operators.\\
Observe that $I=I_1^2\times I_2^2=\mathfrak{L}\{1,x^{\beta}\}\times \mathfrak{L}\{1,x^{\beta}\}$ is invariant under the operator $N[f,g]$ as
\begin{align*}
N_1[k_1+k_2x^{\beta},l_1+l_2x^{\beta}]&=-\Gamma(1+\beta)\left[(a_1k_1k_2+a_2k_1l_2+a_2l_1k_2)-(a_1k_2^2+2a_2k_2l_2)x^{\beta}\right]\in I_1^2,\\
N_2[k_1+k_2x^{\beta},l_1+l_2x^{\beta}]&=-\Gamma(1+\beta)\left[(b_1l_1l_2+b_2k_1l_2+b_2l_1k_2)-(b_1l_2^2+2b_2k_2l_2)x^{\beta}\right]\in I_2^2.
\end{align*}
In view of Theorem \ref{2.THE1}, system (\ref{2.1EX1}) admits solution of the form
\begin{equation}\label{2.1SOL1}
f(t,x)=K_1(t)+K_2(t)x^{\beta}, g(t,x)=L_1(t)+L_2(t)x^{\beta},
\end{equation}where $K_1(t), K_2(t),L_1(t)$ and $L_2(t)$ satisfy the following system of FODEs
\begin{align}
\dfrac{^{RL} d ^\alpha K_1(t)}{ d t^ {\alpha}}&=-\Gamma(1+\beta)[a_1K_1(t)K_2(t)+a_2K_1(t)L_2(t)+a_2L_1(t)K_2(t)],\label{2.1EX1.1}\\
\dfrac{^{RL} d ^\alpha K_2(t)}{ d t^ {\alpha}}&=-\Gamma(1+\beta)[a_1K_2^2(t)+2a_2K_2(t)L_2(t)],\label{2.1EX1.2}\\
\dfrac{^{RL} d ^\alpha L_1(t)}{ d t^ {\alpha}}&=-\Gamma(1+\beta)[b_1L_1(t)L_2(t)+b_2K_1(t)L_2(t)+b_2L_1(t)K_2(t)],\label{2.1EX1.3}\\
\dfrac{^{RL} d ^\alpha L_2(t)}{ d t^ {\alpha}}&=-\Gamma(1+\beta)[b_1L_2^2(t)+2b_2K_2(t)L_2(t)].\label{2.1EX1.4}
\end{align}
Solving Eq. (\ref{2.1EX1.2}) and Eq. (\ref{2.1EX1.4}), we obtain
 \begin{equation}\label{2.1EX1.5}
L_2(t)=M_2t^{-\alpha},~~ K_2(t)=\frac{-b_1}{2b_2}M_2t^{-\alpha}-\frac{\Gamma(1-\alpha)t^{-\alpha}}{2b_2\Gamma(1-2\alpha)\Gamma(1+\beta)},~M_2(\neq 0)~\text{ is arbitrary}.
\end{equation}
Using (\ref{2.1EX1.5}), and solving Eq. (\ref{2.1EX1.1}) and Eq. (\ref{2.1EX1.3}), we deduce the solution of the system of generalized fractional Burger's equations (\ref{2.1EX1}) as
\begin{align*}
f(t,x)&=\frac{-M_1\Gamma(1-\alpha)t^{-\alpha}}{2b_2M_2\Gamma(1-2\alpha)\Gamma(1+\beta)}-\frac{b_1M_1t^{-\alpha}}{2b_2}+\left[\frac{-b_1M_2t^{-\alpha}}{2b_2}-\frac{\Gamma(1-\alpha)t^{-\alpha}}{2b_2\Gamma(1-2\alpha)\Gamma(1+\beta)}\right]x^{\beta},\\
g(t,x)&=M_1t^{-\alpha}+[M_2t^{-\alpha}]x^{\beta}, ~\text{where}~a_1,a_2,b_1,b_2,M_1, M_2(\neq 0)~ \text{are arbitrary}.
\end{align*}\\
Note that in particular, for $\alpha=\beta=1,~a_1=b_1=-2,~ a_0=b_0=-1$ and $a_2=b_2=1$ Eq. (\ref{2.1EX1}) becomes coupled Burger equation \cite{srivastava2014one}.\\
Now consider fractional version of coupled Burger equation. 
	\begin{align}\label{2.1EX1.7}
\frac{^{RL}\partial^{\alpha}f}{\partial t^{\alpha}}&=\frac{\partial^{2\beta}f}{\partial x^{2\beta}}+2f\left(\frac{\partial^{\beta}f}{\partial x^{\beta}}\right)-\left(f\frac{\partial^{\beta}g}{\partial x^{\beta}}+g\frac{\partial^{\beta}f}{\partial x^{\beta}}\right),\nonumber\\
\frac{^{RL}\partial^{\alpha}g}{\partial t^{\alpha}}&=\frac{\partial^{2\beta}g}{\partial x^{2\beta}}+2g\left(\frac{\partial^{\beta}g}{\partial x^{\beta}}\right)-\left(f\frac{\partial^{\beta}g}{\partial x^{\beta}}+g\frac{\partial^{\beta}f}{\partial x^{\beta}}\right),
\end{align}Using  $a_1=b_1=-2,~ a_2=b_2=1$ and Eq. (\ref{2.1EX1.1}) we find the value of $M_2$ as \begin{equation*}
M_2=\frac{-\Gamma(1-\alpha)}{2\Gamma(1-2\alpha)\Gamma(1+\beta)}.
\end{equation*}
Thus
\begin{align}\label{2.1EX1.8}
K_1(t)=2M_1t^{-\alpha}, K_2(t)=\frac{-\Gamma(1-\alpha)t^{-\alpha}}{\Gamma(1-2\alpha)\Gamma(1+\beta)},~L_1(t)=M_1t^{-\alpha}, L_2(t)=\frac{-\Gamma(1-\alpha)t^{-\alpha}}{2\Gamma(1-2\alpha)\Gamma(1+\beta)}.
\end{align}
From (\ref{2.1SOL1}) and (\ref{2.1EX1.8}) we deduce the solution of fractional version of coupled Burger equations (\ref{2.1EX1.7}) as
\begin{align}\label{2.1SOL11}
f(t,x)=2M_1t^{-\alpha}+\left[\frac{-\Gamma(1-\alpha)t^{-\alpha}}{\Gamma(1-2\alpha)\Gamma(1+\beta)}\right]x^{\beta},~~
g(t,x)=M_1t^{-\alpha}+\left[\frac{-\Gamma(1-\alpha)t^{-\alpha}}{2\Gamma(1-2\alpha)\Gamma(1+\beta)}\right]x^{\beta},~M_1\in\mathbb{R}.
\end{align} 
The solution (\ref{2.1SOL11}) is depicted in Fig. 1.
	\begin{figure}[H]
	\centering
	\includegraphics[scale=0.4]{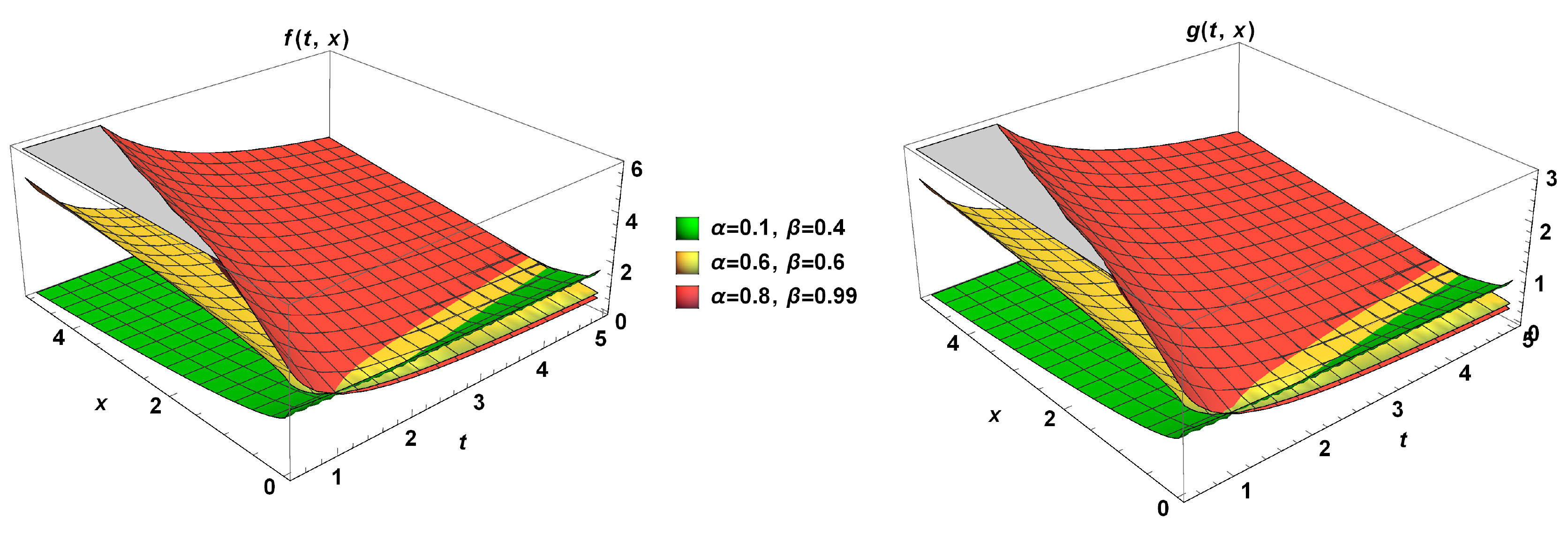}
	\caption{Plot of $f(t,x)$ and $g(t,x)$ for $M_1=1$, and different values of $\alpha$ and $~\beta.$}
\end{figure}
\subsubsection{Exact solution of coupled FPDEs}
	Consider the following system of nonlinear FPDEs for $t>0, 0< \alpha_1,~\alpha_2,~\beta \leq 1.$
\begin{align}\label{2.1EX2}
	\frac{\partial^{\alpha_1}f}{\partial t^{\alpha_1}}+\frac{\partial^{\alpha_1+1}f}{\partial t^{\alpha_1+1}}&=\frac{\partial^{2\beta}f}{\partial x^{2\beta}}+m_1\left(g\frac{\partial^{\beta}g}{\partial x^{\beta}}\right)+a_1m_1g^2,\nonumber \\
	\frac{\partial^{\alpha_2}g}{\partial t^{\alpha_2}}+\frac{\partial^{\alpha_2+1}g}{\partial t^{\alpha_2+1}}&=\frac{\partial^{2\beta}g}{\partial x^{2\beta}}+ n_1\frac{\partial^{2\beta}f}{\partial x^{2\beta}}+a_2^2n_1f+n_2g,
\end{align}	where $a_1,a_2,m_1,m_2,n_1$ and $n_2$ are arbitrary constants. \\
In view of (\ref{2.SFPDE}),
\begin{align*}
N_1[f,g]&=\frac{\partial^{2\beta}f}{\partial x^{2\beta}}+m_1\left(g\frac{\partial^{\beta}g}{\partial x^{\beta}}\right)+a_1m_1g^2,\\
N_2[f,g]&=\frac{\partial^{2\beta}g}{\partial x^{2\beta}}+ n_1\frac{\partial^{2\beta}f}{\partial x^{2\beta}}+a_2^2n_1f+n_2g.
\end{align*}
Clearly $I=I_1^2\times I_2^1=\mathfrak{L}\{\sin_{\beta}(a_2x^{\beta}),\cos_{\beta}(a_2x^{\beta})\}\times \mathfrak{L}\{E_{\beta}(-a_1x^{\beta})\}$ is an invariant subspace under $N[f,g]$ as
\begin{align*}
N_1\left[k_1\sin_{\beta}(a_2x^{\beta})+k_2\cos_{\beta}(a_2x^{\beta}),l_1E_{\beta}(-a_1x^{\beta})\right]&=-k_1a_2^2\sin_{\beta}(a_2x^{\beta})-k_2a_2^2\cos_{\beta}(a_2x^{\beta})\in I_1^2,\\
N_2\left[k_1\sin_{\beta}(a_2x^{\beta})+k_2\cos_{\beta}(a_2x^{\beta}),l_1E_{\beta}(-a_1x^{\beta})\right]&=(a_1^2+n_2)l_1E_{\beta}(-a_1x^{\beta})\in I_2^1.
\end{align*}
Thus Theorem \ref{2.THE1} implies that the system (\ref{2.1EX2}) admits solution of the form
\begin{equation}\label{2.1SOL2.1}
f(t,x)=K_1(t)\sin_{\beta}(a_2x^{\beta})+K_2(t)\cos_{\beta}(a_2x^{\beta}), ~g(t,x)=L_1(t)E_{\beta}(-a_1x^{\beta}),
\end{equation}where $K_1(t), K_2(t)$ and $L_1(t)$ satisfy the system of FODEs
\begin{align}
\dfrac{ d ^{\alpha_1} K_1(t)}{ d t^ {\alpha_1}}+\dfrac{ d ^{\alpha_1+1} K_1(t)}{ d t^ {\alpha_1+1}}&=-a_2^2K_1(t),\label{2.1EX2.1}\\
\dfrac{ d ^{\alpha_1} K_2(t)}{ d t^ {\alpha_1}}+\dfrac{ d ^{\alpha_1+1} K_2(t)}{ d t^ {\alpha_1+1}}&=-a_2^2K_2(t),\label{2.1EX2.2}\\
\dfrac{ d ^{\alpha_2} L_1(t)}{ d t^ {\alpha_2}}+\dfrac{ d ^{\alpha_2+1} L_1(t)}{ d t^ {\alpha_2+1}}&=(a_1^2+n_2)L_1(t)=aL_1(t), ~~~~a=(a_1^2+n_2).\label{2.1EX2.3}
\end{align}
We apply Laplace transform to Eq. (\ref{2.1EX2.1}) and obtain
\begin{align*}
\tilde{K}_1(s)&=\frac{K_1(0)s^{\alpha_1}}{s^{\alpha_1}+s^{\alpha_1 +1}+a_2^2}+\frac{[K_1(0)+K_1'(0)]s^{\alpha_1-1}}{s^{\alpha_1}+s^{\alpha_1 +1}+a_2^2}\\
&=K_1(0)\sum_{m=0}^{\infty}\frac{(-1)^ma_2^{2m}s^{-\alpha_1m}}{(s+1)^{m+1}}+\left[K_1(0)+K_1'(0)\right]\sum_{m=0}^{\infty}\frac{(-1)^ma_2^{2m}s^{-\alpha_1m-1}}{(s+1)^{m+1}},~\mathcal{R}e(s)>1.
\end{align*}
Taking inverse Laplace transform and using the relation (\ref{2.P1})
\begin{align*}
K_1(t)=K_1(0)\sum_{m=0}^{\infty}\frac{(-1)^m}{m!}a_2^{2m}t^{(\alpha_1+1)m}E_{1,{\alpha}_1m+1}^{(m)}(-t)+\left[K_1(0)+K_1'(0)\right]\sum_{m=0}^{\infty}\frac{(-1)^m}{m!}a_2^{2m}t^{(\alpha_1+1)m+1}E_{1,{\alpha_1}m+2}^{(m)}(-t).
\end{align*}
Proceeding on similar lines we evaluate $K_2(t)$ and $L_2(t)$. Hence an exact solution of the system (\ref{2.1EX2}) is
\begin{align*}
f(t,x)&=\left[K_1(0)\sum_{m=0}^{\infty}\frac{(-1)^m}{m!}a_2^{2m}t^{(\alpha_1+1)m}E_{1,{\alpha}_1m+1}^{(m)}(-t)+\left[K_1(0)+K_1'(0)\right]\right.\\
&\left.\sum_{m=0}^{\infty}\frac{(-1)^m}{m!}a_2^{2m}t^{(\alpha_1+1)m+1}E_{1,{\alpha_1}m+2}^{(m)}(-t)\right]\sin_{\beta}(a_2x^{\beta})+\left[K_2(0)\sum_{m=0}^{\infty}\frac{(-1)^m}{m!}a_2^{2m}t^{(\alpha_1+1)m}\right.\\
&\times \left.E_{1,{\alpha}_1m+1}^{(m)}(-t)+\left[K_2(0)+K_2'(0)\right]\sum_{m=0}^{\infty}\frac{(-1)^m}{m!}a_2^{2m}t^{(\alpha_1+1)m+1}E_{1,{\alpha_1}m+2}^{(m)}(-t)\right]\cos_{\beta}(a_2x^{\beta}),\\
g(t,x)&=\left[L_1(0)\sum_{m=0}^{\infty}\frac{a^{m}}{m!}
t^{(\alpha_2+1)m}E_{1,{\alpha}_2m+1}^{(m)}(-t)+\left[L_1(0)+L_1'(0)\right]\right.\\
&\times \left.\sum_{m=0}^{\infty}\frac{a^{m}}{m!}t^{(\alpha_2+1)m+1}E_{1,{\alpha_2}m+2}^{(m)}(-t)\right]E_{\beta}(-a_1x^{\beta}),
\end{align*}where $a=a_1^2+n_2$ and $a_1,a_2,n_2$ are arbitrary.

\subsubsection{Coupled fractional Boussinesq equations}
Fractional version of coupled Boussinesq equations along with initial conditions for $t>0,~0< \alpha_1,~\alpha_2, ~\beta \leq 1$ is
\begin{align}
&\frac{\partial^{\alpha_1}f}{\partial t^{\alpha_1}}=\frac{-\partial^{\beta}g}{\partial x^{\beta}}=N_1[f,g],\label{2.1EX30}\\
&\frac{\partial^{\alpha_2}g}{\partial t^{\alpha_2}}=-m_1\frac{\partial^{\beta}f}{\partial x^{\beta}}+3f\left(\frac{\partial^{\beta}f}{\partial x^{\beta}}\right)+ m_2\frac{\partial^{3\beta}f}{\partial x^{3\beta}}=N_2[f,g],\label{2.1EX31}\\
&f(0,x)= e+2x^{\beta},~g(0,x)=\frac{3}{2},\label{2.1IC3}
\end{align}	where $m_1$, $m_2$ are arbitrary constants.\\
$I=I_1^2\times I_2^2=\mathfrak{L}\{1,x^{\beta}\}\times \mathfrak{L}\{1,x^{\beta}\}$ is invariant subspace w.r.t. the operator $N[f,g]$ as
\begin{align*}
N_1[k_1+k_2x^{\beta},l_1+l_2x^{\beta}]&=-\Gamma(1+\beta)l_2\in I_1^2,\\
N_2[k_1+k_2x^{\beta},l_1+l_2x^{\beta}]&=\Gamma(1+\beta)\left[-m_1k_2+3k_1k_2\right]+3\Gamma(1+\beta)k_2^2x^{\beta}\in I_2^2.
\end{align*}
Hence using Theorem \ref{2.THE1}, system (\ref{2.1EX30})-(\ref{2.1EX31}) has the following solution
\begin{equation}\label{2.1SOL3}
f(t,x)=K_1(t)+K_2(t)x^{\beta}, g(t,x)=L_1(t)+L_2(t)x^{\beta},
\end{equation}where the unknowns functions $K_1(t), K_2(t),L_1(t)$ and $L_2(t)$ satisfy the following system of FODEs
\begin{align}
\dfrac{ d ^{\alpha_1} K_1(t)}{ d t^ {\alpha_1}}&=-\Gamma(1+\beta)L_2(t),\label{2.1EX3.1}\\
\dfrac{ d ^{\alpha_1} K_2(t)}{ d t^ {\alpha_1}}&=0,\label{2.1EX3.2}\\
\dfrac{ d ^{\alpha_2} L_1(t)}{ d t^ {\alpha_2}}&=\Gamma(1+\beta)\left[-m_1K_2(t)+3K_1(t)K_2(t)\right],\label{2.1EX3.3}\\
\dfrac{ d ^{\alpha_2} L_2(t)}{ d t^ {\alpha_2}}&=3\Gamma(1+\beta)K_2^2(t).\label{2.1EX3.4}
\end{align}
Eq. (\ref{2.1EX3.2}) implies that $K_2(t)=b$ (constant). Hence substituting value of $K_2(t)$ in Eq. (\ref{2.1EX3.4}) and performing fractional integration on both sides we obtain
 $L_2(t)=L_2(0)+3b^2\dfrac{\Gamma(1+\beta)t^{\alpha_2}}{\Gamma(1+\alpha_2)}$. Proceeding on similar lines solution (\ref{2.1SOL3}) takes the following form
\begin{align}\label{2.1SOL31}
f(t,x)&=\left[a-d\frac{\Gamma(1+\beta)}{\Gamma(1+\alpha_1)}t^{\alpha_1}-3b^2\frac{\Gamma(1+\beta)^2t^{\alpha_1+\alpha_2}}{\Gamma(1+\alpha_1+\alpha_2)}\right]+bx^{\beta},\nonumber\\
g(t,x)&=\left[c-m_1b\frac{\Gamma(1+\beta)}{\Gamma(1+\alpha_2)}t^{\alpha_2}+3ba\frac{\Gamma(1+\beta)}{\Gamma(1+\alpha_2)}t^{\alpha_2}-3bd\frac{\Gamma(1+\beta)^2t^{\alpha_1+\alpha_2}}{\Gamma(1+\alpha_1+\alpha_2)}-9b^3\frac{\Gamma(1+\beta)^3t^{\alpha_1+2\alpha_2}}{\Gamma(1+\alpha_1+2\alpha_2)}\right]\nonumber\\
&+\left[d+3b^2\dfrac{\Gamma(1+\beta)t^{\alpha_2}}{\Gamma(1+\alpha_2)}\right]x^{\beta},~a,b, c~\text{and}~d~\text{are arbitrary.}
\end{align}
\textbf{Note:} In particular for $\alpha_1= \alpha_2=\alpha,$ and $ \beta=1$ the solution has the form 
\begin{align*}
f(t,x)&=\left[a-\frac{dt^{\alpha}}{\Gamma(1+\alpha)}-\frac{3b^2t^{2\alpha}}{\Gamma(1+2\alpha)}\right]+bx,\\
g(t,x)&=\left[c-\frac{m_1bt^{\alpha}}{\Gamma(1+\alpha)}+\frac{3bat^{\alpha}}{\Gamma(1+\alpha)}-\frac{3bdt^{2\alpha}}{\Gamma(1+2\alpha)}-\frac{9b^3t^{3\alpha}}{\Gamma(1+3\alpha)}\right]+\left[d+\dfrac{3b^2t^{\alpha}}{\Gamma(1+\alpha)}\right]x.
\end{align*}This is a solution of the time fractional coupled Boussinesq equation obtained by Sahadevan \textit{et. al} \cite{sahadevan2017exact}. Exact solution of the system (\ref{2.1EX30})-(\ref{2.1EX31}) along with the initial conditions (\ref{2.1IC3}) is
\begin{align}\label{2.1SOL32}
f(t,x)&=\left[e-\frac{12\Gamma(1+\beta)^2}{\Gamma(1+\alpha_1+\alpha_2)}t^{\alpha_1+\alpha_2}\right]+2x^{\beta},\nonumber\\
g(t,x)&=\left[\dfrac{3}{2}-\frac{2m_1\Gamma(1+\beta)}{\Gamma(1+\alpha_2)}t^{\alpha_2}+\frac{6e\Gamma(1+\beta)}{\Gamma(1+\alpha_2)}t^{\alpha_2}-\frac{72\Gamma(1+\beta)^3t^{\alpha_1+2\alpha_2}}{\Gamma(1+\alpha_1+2\alpha_2)}\right]+\left[\dfrac{12\Gamma(1+\beta)}{\Gamma(1+\alpha_2)}t^{\alpha_2}\right]x^{\beta}.
\end{align}
The solution (\ref{2.1SOL32}) is plotted in Fig. 2.
	\begin{figure}[H]
	\centering
	\includegraphics[scale=0.4]{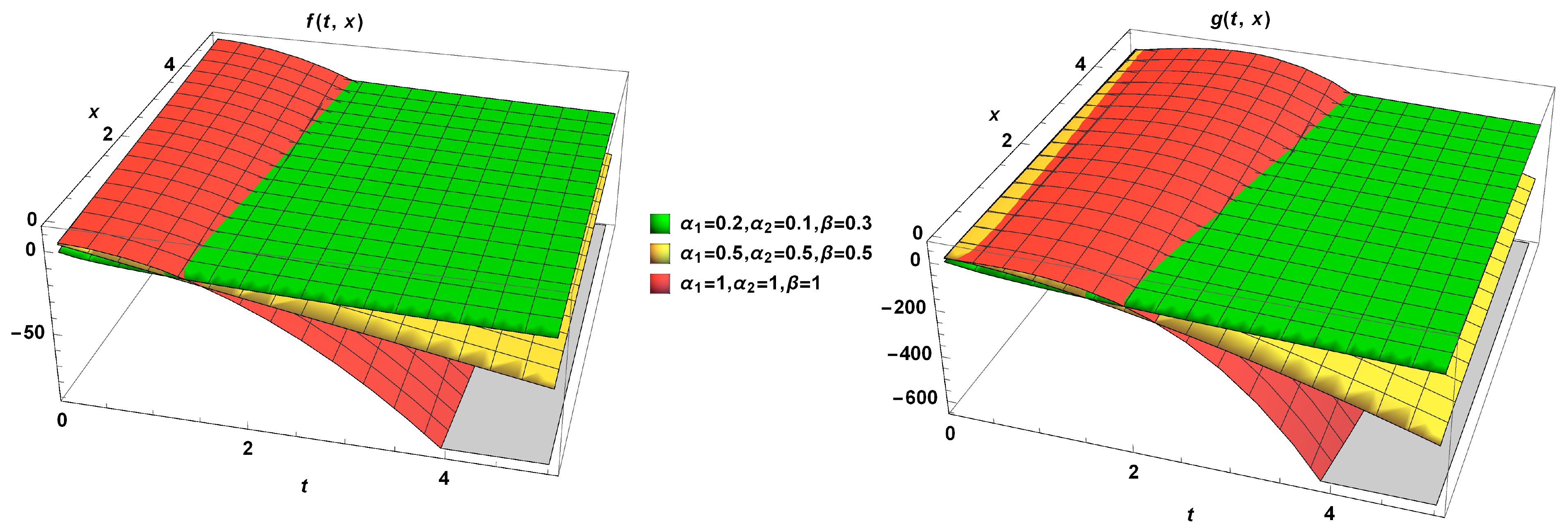}
	\caption{Plot of the solution (\ref{2.1SOL32}) for various values of $\alpha_1,~\alpha_2$ and $\beta.$}
\end{figure}
\subsubsection{Exact solution of fractional version of system of KdV type of equations}
	Consider fractional KdV type of fractional coupled equations for $t>0,~\alpha \in (0,1]\backslash\{\frac{1}{2}\},~\beta\in (0,1].$
		\begin{align}\label{2.1EX4}
			\frac{^{RL}\partial^{\alpha}f}{\partial t^{\alpha}}&=a_1\left(f\frac{\partial^{\beta}f}{\partial x^{\beta}}\right)+a_2\left(g\frac{\partial^{\beta}g}{\partial x^{\beta}}\right)+a_3\frac{\partial^{3\beta}f}{\partial x^{3\beta}},\nonumber\\
			\frac{^{RL}\partial^{\alpha}g}{\partial t^{\alpha}}&=b_1\left(f\frac{\partial^{\beta}g}{\partial x^{\beta}}\right)+b_2\left(g\frac{\partial^{\beta}f}{\partial x^{\beta}}\right)+b_3\frac{\partial^{3\beta}g}{\partial x^{3\beta}}.
		\end{align}Here $a_1, a_2, a_3, b_1, b_2$ and $b_3$ are arbitrary constants such that $b>a_1$ and $a_2\geq 0.$ Note that for $\alpha=1=\beta,$ and
		\begin{itemize}
			\item When $a_1,a_2,a_3,b_1,b_3$ are arbitrary constants and $b_2=0$ then the coupled system (\ref{2.1EX4}) is the coupled KdV system given in ref. \cite{sahadevan2017lie}.
			\item When $a_1=a_2=b_1=b_2=6,$ and $a_3=1=b_3,$ then the system (\ref{2.1EX4}) is well-known complex coupled KdV system studied in \cite{inan2007exact}.
			\item When $a_1=6a,~ a_2=2b,~a_3=a,~b_1=-3,~b_2=0$ and $b_3=-1,$ ($a,b$ are arbitrary), then the system (\ref{2.1EX4}) reduces to Hirota-Satsuma (HS)-KdV system proposed by Hirota ans Satsuma in 1981 to model interactions of two long waves with different dispersion relations \cite{hirota1981soliton}.
			\item When $a_1=6,a_2=2,a_3=1, b_1=b_2=2$ and $b_3=0,$ then the KdV type system (\ref{2.1EX4}) is treated as Ito type coupled KdV system \cite{ito1982symmetries}.
		\end{itemize} Further note when $g=0$, and $a_1=6, a_2=0,a_3=-1$ the fractional KdV system (\ref{2.1EX4}) reduces to fractional KdV equation studied by Choudhary and Daftardar-Gejji \cite{choudhary2017invariant}.\\In system (\ref{2.1EX4})
	\begin{align*}
	N_1[f,g]&=a_1\left(f\frac{\partial^{\beta}f}{\partial x^{\beta}}\right)+a_2\left(g\frac{\partial^{\beta}g}{\partial x^{\beta}}\right)+a_3\frac{\partial^{3\beta}f}{\partial x^{3\beta}},\\
	N_2[f,g]&=b_1\left(f\frac{\partial^{\beta}g}{\partial x^{\beta}}\right)+b_2\left(g\frac{\partial^{\beta}f}{\partial x^{\beta}}\right)+b_3\frac{\partial^{3\beta}g}{\partial x^{3\beta}}.
	\end{align*}
	Note that $I=I_1^2\times I_2^2=\mathfrak{L}\{1,x^{\beta}\}\times \mathfrak{L}\{1,x^{\beta}\}$ is invariant subspace with respect to the operator $N[f,g]$ since
		\begin{align*}
		N_1[k_1+k_2x^{\beta},l_1+l_2x^{\beta}]&=\Gamma(1+\beta)\left[a_1k_1k_2+a_2l_1l_2\right]+\Gamma(1+\beta)\left[a_1k_2^2+a_2l_2^2\right]x^{\beta} \in I_1^2,\\
			N_2[k_1+k_2x^{\beta},l_1+l_2x^{\beta}]&=\Gamma(1+\beta)\left[b_1k_1l_2+b_2l_1k_2\right]+\Gamma(1+\beta)\left[(b_1+b_2)k_2l_2\right]x^{\beta}\in I_2^2.
		\end{align*}
		Hence the system (\ref{2.1EX4}) admits solution of the form
		\begin{equation}\label{2.1SOL4}
		f(t,x)=K_1(t)+K_2(t)x^{\beta}, g(t,x)=L_1(t)+L_2(t)x^{\beta},
		\end{equation}where the unknown functions $K_1(t), K_2(t),L_1(t)$ and $L_2(t)$ satisfy the following system of FODEs.
		\begin{align}
		\dfrac{^{RL} d ^\alpha K_1(t)}{ d t^ {\alpha}}&=\Gamma(1+\beta)\left[a_1K_1(t)K_2(t)+a_2L_1(t)L_2(t)\right],\label{2.1EX4.1}\\
		\dfrac{^{RL} d ^\alpha K_2(t)}{ d t^ {\alpha}}&=\Gamma(1+\beta)[a_1K_2^2(t)+a_2L_2^2(t)],\label{2.1EX4.2}\\
		\dfrac{^{RL} d ^\alpha L_1(t)}{ d t^ {\alpha}}&=\Gamma(1+\beta)[b_1K_1(t)L_2(t)+b_2L_1(t)K_2(t)],\label{2.1EX4.3}\\
		\dfrac{^{RL} d ^\alpha L_2(t)}{ d t^ {\alpha}}&=\Gamma(1+\beta)[bK_2(t)L_2(t)],~b~=(b_1+b_2).\label{2.1EX4.4}
		\end{align}
	Solving system (\ref{2.1EX4.1})-(\ref{2.1EX4.4}), we deduce the following solution of system (\ref{2.1EX4}):\\
	For $\alpha=1,$
		\begin{align*}
		f(t,x)&=\frac{-\sqrt{a_2}M_1}{\sqrt{b-a_1}t}-\left(\frac{1}{b\Gamma(1+\beta)t}\right)x^{\beta},\\
		g(t,x)&=\dfrac{M_1}{t}+\left(\frac{\sqrt{b-a_1}}{b\sqrt{a_2}\Gamma(1+\beta)t}\right)x^{\beta},
		\end{align*}
	For $\alpha\in (0,1)\backslash\{\frac{1}{2}\}$,
		\begin{align*}
		f(t,x)&=\frac{\sqrt{a_2}}{\sqrt{b-a_1}}M_1t^{-\alpha}+\left[\frac{\Gamma(1-\alpha)t^{-\alpha}}{b\Gamma(1-2\alpha)\Gamma(1+\beta)}\right]x^{\beta},\\
		g(t,x)&=M_1t^{-\alpha}+\left[\frac{\sqrt{b-a_1}~\Gamma(1-\alpha)t^{-\alpha}}{b\sqrt{a_2}~ \Gamma(1-2\alpha)\Gamma(1+\beta)}\right]x^{\beta},
		\end{align*}
				where $M_1, a_1, a_2, b_1$ and $b_2$ are arbitrary such that $b=b_1+b_2, b>a_1,a_2\geq0.$ \\
	Solution of the system (\ref{2.1EX4}) for $\alpha \in (0,1]\backslash\{\frac{1}{2}\},~\beta\in (0,1]$ is plotted in Fig. 3.
	\begin{figure}[H]
	\centering
	\includegraphics[scale=0.5]{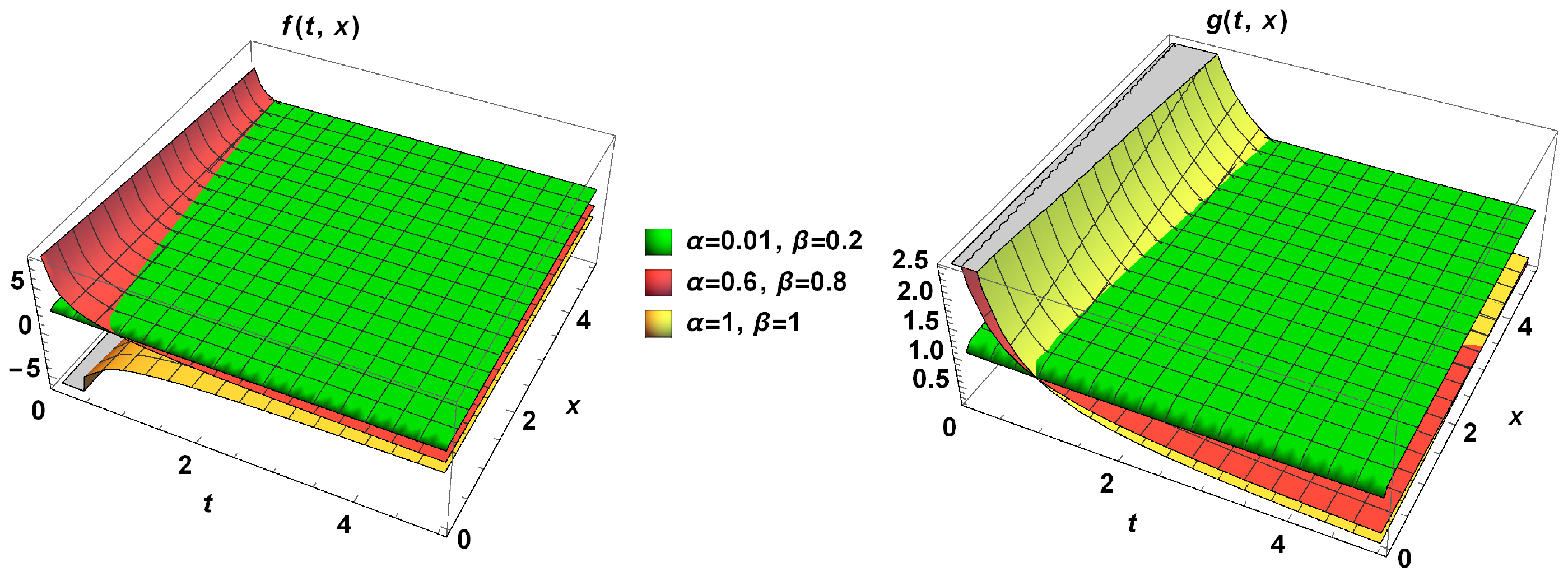}
	\caption{Plot of the solution of the system (\ref{2.1EX4}) for $M_1=1, a_1=2, a_2=4, b=3.$}
\end{figure}
		
\section{FPDEs in (1+n) dimension}
In this section we consider higher dimensional FPDEs of the form
\begin{align}\label{2.HFPDE}
\sum_{i=1}^{r}\lambda _{i}\frac{\partial ^{\gamma_{(i)}}f(t,\bar{x})}{\partial t^{\gamma_{(i)}}}=N^l[f(t,\bar{x})],\end{align}
where$~\bar{x}=(x_1,x_2,\dots,x_n),~l=1,2,~\gamma(i)=i\alpha~ \text{or}~ \gamma(i)=\alpha+i-1.\\
N^1[f(t,\bar{x})]=\hat{N}^1\left [\bar{x}, f,\frac{\partial^{\beta_1}f}{\partial x_1^{\beta_1}}, \dots,\frac{\partial^{\beta_n}f}{\partial x_n^{\beta}}, \frac{\partial^{2\beta_1}f}{\partial x_1^{2\beta_1}},\dots,\frac{\partial^{k\beta{n-1}}f}{\partial x_{n-1}^{k\beta_{n-1}}}, \frac{\partial^{k\beta_n}f}{\partial x_n^{k\beta_n}}\right ],$ and\\
$
N^2[f(t,\bar{x})]=\hat{N}^2\left [\bar{x}, f,\frac{\partial^{\beta_1}f}{\partial x_1^{\beta_1}},\dots, \frac{\partial^{\beta_n}f}{\partial x_n^{\beta_n}},\frac{\partial^{\beta_1+1}f}{\partial x_1^{\beta_1+1}},\dots, \frac{\partial^{\beta_{n-1}+k-1}f}{\partial x_{n-1}^{\beta_{n-1}+k-1}}, \frac{\partial^{\beta_n+k-1}f}{\partial x_n^{\beta_n+k-1}}\right]
$
	are nonlinear fractional differential operators in higher dimensions. Here $\frac{\partial^{\gamma_{(i)}}f(\cdot)}{\partial t^{\gamma_{(i)}}}$ is Caputo (or Riemann-Liouville) derivative with respect to $t$. $\frac{\partial^{j\beta_i}f(\cdot)}{\partial x_i^{j\beta_i}}$ and $\frac{\partial^{\beta_i+j-1}f(\cdot)}{\partial x_i^{\beta_i+j-1}}, i=1,\dots,n,~ j=1,\dots,k ~(k\in \mathbb{N})$ are Caputo derivatives with respect to variable $x_i$. $\lceil \alpha \rceil=s, \lceil \beta_i \rceil=s_i,$ where $s, s_i \in \mathbb{N}, i=1,\dots k$ and $\lambda_i\in \mathbb{R}$.
\subsection{Invariant subspace method FPDEs in higher dimensions}
Let $I^m$ denote the $m$-dimensional linear space over $\mathbb{R}$ spanned by $m$ linearly independent basis functions $\{\phi_j(x_1,x_2,\dots,x_n):~ j=1,\dots,m\}$, \textit{i.e.},
\begin{align*}
	I^m= \mathfrak{L}\{\phi_1(\bar{x}),\phi_2(\bar{x}),\dots,\phi_{m}(\bar{x})\}=\left \{\sum_{j=1}^{m}k_j\phi_j(\bar{x})~ \bigg| ~k_j\in \mathbb{R}, i=1,\dots,m \right \}.
\end{align*}
A finite dimensional linear space $I^m$ is said to be invariant under fractional differential operators $N^l[f(t,\bar{x})] (l=1,2),$ if $N^l[f]\in I^m, \forall f \in I^m.$
\begin{The}\label{2.THE2}
	If a finite dimensional linear space $I^{m}$ is invariant under the operators $N^l[f(t,\bar{x})], l=1,2$, then FPDE (\ref{2.HFPDE}) has a solution of the form
	\begin{equation}\label{2.THE2.1}
	f(t,\bar{x})=\sum_{j=1}^{m}K_j(t)\phi_j(\bar{x}),
	\end{equation}where the coefficient $\{K_j\}'s$ satisfy the following system of FODEs
	\begin{equation}\label{2.THE2.2}
	\sum_{i=1}^{r}\lambda_i\frac{ d^{\gamma(i)}K_j(t)}{ dt^{\gamma(i)}}=\psi_j(K_1(t),K_2(t),\dots,K_m(t)),~j=1,\dots,m.
	\end{equation}Here $~\gamma(i)=i\alpha~ \text{or}~ \gamma(i)=\alpha+i-1$, and $\{\psi_j\}'s$ are the expansion coefficients of $N^l[f(t,\bar{x})], l=1, 2$ with respect to basis function $\{\phi_j\}'s$ of $I^m$.
\end{The}
\begin{proof}
	 Using linearity of fractional derivatives and Eq. (\ref{2.THE2.1}), L.H.S of FPDE (\ref{2.HFPDE}) reduces to 
	\begin{equation}\label{2.LHS2.1}
	\sum_{i=1}^{r}\lambda_i\frac{\partial ^{\gamma(i)}f(t,\bar{x})}{\partial t^{\gamma(i)}}=\sum_{i=1}^{r}\lambda_i\frac{\partial^{\gamma(i)}}{\partial t^{\gamma(i)}}\left(\sum_{j=1}^{m}K_j(t) \phi_j(\bar{x})\right)=\sum_{j=1}^{m}\left[\sum_{i=1}^{r}\lambda_i\frac{ d^{\gamma(i)}K_j(t)}{ dt^{\gamma(i)}}\right] \phi_j(\bar{x}).
	\end{equation}	
	Further as $I^m$ is an invariant space under the operator $N^l[f]$, there exist $m$ linearly independent functions $\phi_1(\bar{x}),\phi_2(\bar{x}),\dots,\phi _m(\bar{x})$ such that \begin{equation}\label{2.IN2.1}
N^l\left[\sum_{j=1}^{m}k_j \phi_j(\bar{x})\right]=\sum_{j=1}^{m}\psi_j(k_1,k_2,\dots,k_m)\phi_j(\bar{x}),~\text{for}~ k_j\in \mathbb{R}, l=1,2,\end{equation}
	where $\{\psi_j\}'s$ are expansion coefficients of $N^l[f]\in I^{m}$ with respect to the basis $\{\phi_j\}_{j=1}^{m}$.\\
	In view of Eq. (\ref{2.THE2.1}) and Eq. (\ref{2.IN2.1})
	\begin{equation}\label{2.RHS2.1}
N^l[f(t,\bar{x})]=N^l\left[\sum_{j=1}^{m}K_j(t)\phi_j(\bar{x})\right]=\sum_{j=1}^{m}\psi_j(K_1(t),\dots,K_m(t))\phi_j(\bar{x}), l=1,2.
	\end{equation}
	Substituting Eq. (\ref{2.LHS2.1}) and Eq. (\ref{2.RHS2.1}) in Eq. (\ref{2.HFPDE}), we get
	\begin{equation}\label{2.LHS=RHS2}
	\sum_{j=1}^{m}\left[\sum_{i=1}^{r}\lambda_i\frac{ d^{\gamma(i)}K_j(t)}{ dt^{\gamma(i)}}-\psi_j(K_1(t),K_2(t),\dots,K_m(t))\right]\phi_j(\bar{x})=0.\end{equation}
	Using Eq. (\ref{2.LHS=RHS2}) and the fact that $\{\phi_j\}$'s are basis functions, we get the following  system of FODEs
	\begin{equation*}
	\sum_{i=1}^{r}\lambda_i\frac{ d^{\gamma(i)}K_j(t)}{ dt^{\gamma(i)}}=\psi_j(K_1(t),K_2(t),\dots,K_m(t)),~j=1,\dots,m,
	\end{equation*}
where$	~\gamma(i)=i\alpha~ \text{or}~ \gamma(i)=\alpha+i-1.$
\end{proof}
\subsection{Illustrative examples for FPDEs in higher dimensions}
\subsubsection{Fractional dispersive KdV equation in (1+n) dimensions.}
	Consider the following linear fractional dispersive KdV equation
	\begin{equation}\label{2.2EX1}
	\dfrac{\partial ^{\alpha}f}{\partial t^{\alpha}}+\frac{\partial ^{3\beta_1}f}{\partial x_1^{3\beta_1}}+\frac{\partial ^{3\beta_2}f}{\partial x_2^{3\beta_2}}+\cdots+\frac{\partial ^{3\beta_n}f}{\partial x_n^{3\beta_n}}=0,~t>0,~\beta_1,~\beta_2, \cdots,~\beta_n \in (0,1].
	\end{equation}
	\noindent In view of FPDE (\ref{2.HFPDE}), we note that  $N[f]=-\frac{\partial ^{3\beta_1}f}{\partial x_1^{3\beta_1}}-\cdots-\frac{\partial ^{3\beta_n}f}{\partial x_n^{3\beta_n}}$. Observe that when\\
	 $I^{2n}=\mathfrak{L}\{\cos_{\beta_1}(\lambda_1 x_1^{\beta_1}),\sin_{\beta_1}(\lambda_1 x_1^{\beta_1}),\cos_{\beta_2}(\lambda_2 x_2^{\beta_2}),\sin_{\beta_2}(\lambda_2 x_2^{\beta_2}),\cdots,\cos_{\beta_n}(\lambda_n x_n^{\beta_n}),\sin_{\beta_n}(\lambda_n x_n^{\beta_n})\}$,
	\begin{align*}
	N\left[\sum_{i=1}^{n}\left( k_{i1}\cos_{\beta_i}(\lambda_i x_i^{\beta_i})+k_{i2}\sin_{\beta_i}(\lambda_i x_i^{\beta_i})\right)\right]=\sum_{i=1}^{n}\left(-\lambda_i^3k_{i1}\sin_{\beta_i}(\lambda_i x^{\beta_i})+\lambda_i^3k_{i2}\cos_{\beta_i}(\lambda_i x^{\beta_i})\right)\in I^{2n}.
	\end{align*} Hence $I^{2n}$ is an invariant subspace of fractional operator $N[f]$. Hence Theorem (\ref{2.THE2}) implies an exact solution of the form
	\begin{equation}\label{2.2SOL1}
	f(t,\bar{x})=\sum_{i=1}^{n}\left( k_{i1}(t)\cos_{\beta_i}(\lambda_i x_i^{\beta_i})+k_{i2}(t)\sin_{\beta_i}(\lambda_i x_i^{\beta_i})\right),~\lambda_i,~i=1,\dots,n~\text{are distinct.}	\end{equation} where $K_{i1}(t)$ and $ K_{i2}(t),~i=1,\dots,n$ are the unknown functions to be determined by solving the following system of FODEs:
	\begin{multicols}{2}
		\begin{align}
		\dfrac{ d ^{\alpha} K_{11}(t)}{ d t^ {\alpha}}&=\lambda_1^3K_{12}(t),\label{2.2EX1.1}\\
		\dfrac{ d ^{\alpha} K_{12}(t)}{ d t^ {\alpha}}&=-\lambda_1^3K_{11}(t),\label{2.2EX1.2}\\
		\dfrac{ d ^{\alpha} K_{21}(t)}{ d t^ {\alpha}}&=\lambda_2^3K_{22}(t),\label{2.2EX1.3}\\
		\dfrac{ d ^{\alpha} K_{22}(t)}{ d t^ {\alpha}}&=-\lambda_2^3K_{21}(t),
		\end{align}
		
		\columnbreak
		\begin{align}
		&\vdots\nonumber\\
		&\vdots\nonumber\\
			\dfrac{ d ^{\alpha} K_{n1}(t)}{ d t^ {\alpha}}&=\lambda_2^3K_{n2}(t),\\
			\dfrac{ d ^{\alpha} K_{n2}(t)}{ d t^ {\alpha}}&=-\lambda_2^3K_{n1}(t).\label{2.2EX1.4}
			\end{align}
\end{multicols}
	After solving the system $(\ref{2.2EX1.1})-(\ref{2.2EX1.4}),$ we get
\begin{align}\label{2.2EX1.5}
K_{i1}(t)=a_i\sin_{\alpha}(\lambda_i^3t^{\alpha}),~ K_{i2}(t)=a_i\cos_{\alpha}(\lambda_i^3t^{\alpha}).
\end{align}
From (\ref{2.2SOL1}) and (\ref{2.2EX1.5}) we find the following exact solution of system (\ref{2.2EX1}) as
\begin{align}\label{2.2SOL11}
f(t,\bar x)=\sum_{i=1}^{n}\left(a_i\sin_{\alpha}(\lambda_i^3t^{\alpha})\cos_{\beta_i}(\lambda_i x_i^{\beta_i})+a_i\cos_{\alpha}(\lambda_i^3t^{\alpha})\sin_{\beta_i}(\lambda_i x_i^{\beta_i})\right),
\end{align}where $a_i$ and $\lambda_i$
are arbitrary constants for $i=1,2,\dots,n$.\\
The solution (\ref{2.2SOL11}) for $n=2$ is depicted in Fig. 4.
	\begin{figure}[H]
	\centering
	\includegraphics[scale=0.5]{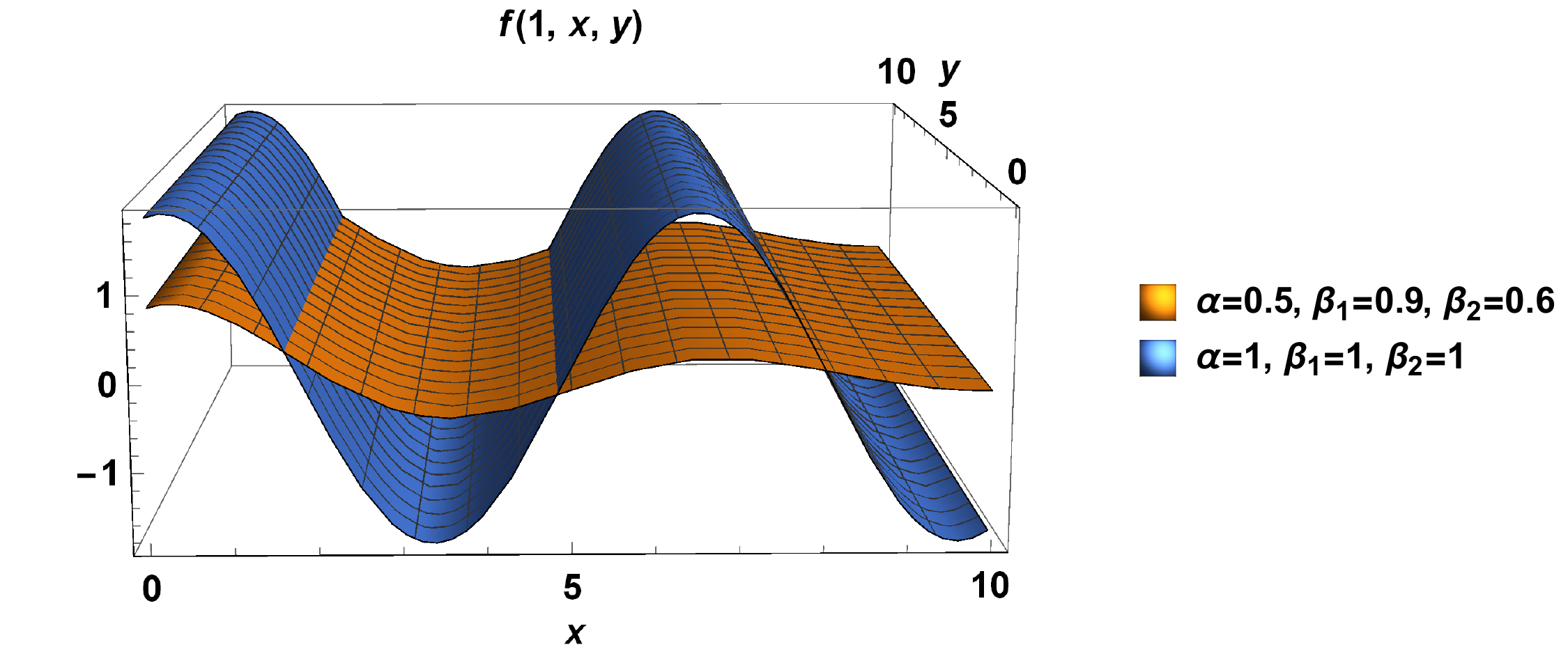}
	\caption{Plot of $f(t,\bar{x})$ of Eq. (\ref{2.2EX1}) for $n=2, a_1=a_2=\lambda_1=1$ and $ \lambda_2=2$ at $t=1$.}
\end{figure}
\textbf{Note:} The fractional dispersion KdV equation admits another invariant subspace
$I^{n}=\mathfrak{L}\{E_{\beta_1}(\lambda_1x_1^{\beta_1}),E_{\beta_2}(\lambda_2x_2^{\beta_2}),\cdots,E_{\beta_n}(\lambda_nx_n^{\beta_n})\}$, leading to following distinct solution:
\begin{equation*}
f(t,\bar x)=\sum_{i=1}^{n}\left[a_iE_{\alpha}(\lambda_i^3t^{\alpha})\right]E_{\beta_i}(\lambda_ix_i^{\beta_i}),~a_i \in \mathbb{R}~\text{for}~i=1,\dots,n.
\end{equation*}
\subsubsection{Fractional version of (1+2) dimensional population model}
	We discuss two dimensional nonlinear FPDE for the population density $f$.
	\begin{equation}\label{2.2EX2}
	\frac{\partial ^{\alpha}f}{\partial t^{\alpha}}=\frac{\partial ^{\beta}}{\partial x^{\beta}}\left(\frac{\partial ^{\beta}f^2}{\partial x^{\beta}}\right)+\frac{\partial ^{\gamma}}{\partial y^{\gamma}}\left(\frac{\partial ^{\gamma}f^2}{\partial y^{\gamma}}\right)+\psi(f),~\alpha, ~\beta,~ \gamma \in (0,1].
	\end{equation}
Note that for $\alpha=\beta=\gamma=1$ and 
\begin{itemize}
	\item $\psi(f)=c f,$ $c\in \mathbb{R}$, the population model $(\ref{2.2EX2})$ follows Malthusian law \cite{gurtin1977diffusion},
	\item $\psi(f)=f(c_1-c_2 f), c_1,c_2\in \mathbb{R}$, Eq. (\ref{2.2EX2}) satisfies Verhulst law \cite{gurtin1977diffusion}.
\end{itemize}
Here we consider $\psi(f)=c f$. Hence $N[f(t,x,y)]=\dfrac{\partial ^{\beta}}{\partial x^{\beta}}\left(\dfrac{\partial ^{\beta}f^2}{\partial x^{\beta}}\right)+\dfrac{\partial ^{\gamma}}{\partial y^{\gamma}}\left(\dfrac{\partial ^{\gamma}f^2}{\partial y^{\gamma}}\right)+c f$ is the fractional nonlinear operator. $I^3=\mathfrak{L}\{1,x^\beta,y^{\gamma}\}$ is invariant under $N[f]$ since
\begin{align*}
N[k_1+k_2x^{\beta}+k_3y^{\gamma}]=ck_1+\Gamma(2\beta+1)k_2^2+\Gamma(2\gamma+1)k_3^2+ck_2x^{\beta}+ck_3y^{\gamma} \in I^3.
\end{align*}
In view of Theorem \ref{2.THE2}, solution of the equation under consideration (\ref{2.1EX2}) is
\begin{equation}\label{2.2SOL2}
f(t,x,y)=K_1(t)+K_2(t)x^{\beta}+K_3(t)y^{\gamma},
\end{equation}where $K_1(t), K_2(t)$ and $K_3(t)$ satisfy the following set of equations:
\begin{align}
\dfrac{ d ^{\alpha} K_1(t)}{ d t^ {\alpha}}&=cK_1(t)+\Gamma(2\beta+1)K_2(t)^2+\Gamma(2\gamma+1)K_3(t)^2,\label{2.2EX2.1}\\
\dfrac{ d ^{\alpha} K_2(t)}{ d t^ {\alpha}}&=c K_2(t),\label{2.2EX2.2}\\
\dfrac{ d ^{\alpha} K_3(t)}{ d t^ {\alpha}}&=c K_3(t)\label{2.2EX2.3}.
\end{align}
Solving Eq. (\ref{2.2EX2.2}) and Eq. (\ref{2.2EX2.3}) using Laplace transform technique, we obtain
\begin{equation*}
K_2(t)=a_2E_\alpha(ct^\alpha),~ K_3(t)=a_3E_\alpha(ct^\alpha),~~a_2, a_3~ \text{are arbitrary}.
\end{equation*}Substituting the obtained values of $K_2(t)$ and $K_3(t)$ in Eq. (\ref{2.2EX2.1}), we get
\begin{equation}\label{2.2EX2.4}
K_1(t)=cK_1(t)+A\left[E_\alpha(ct^\alpha)\right]^2,~~A=a_2^2\Gamma(2\beta+1)+a_3^2\Gamma(2\gamma+1).
\end{equation}We apply NIM \cite{bhalekar2008new} to solve Eq. (\ref{2.2EX2.4}). 
Applying $I^\alpha$ to both sides of Eq. (\ref{2.2EX2.4}), we obtain the following integral equation
\begin{align*}
K_1(t)=a_1+AI^{\alpha}\left[E_\alpha(ct^\alpha)\right]^2+M[K_1(t)],~\text{where}~M[K_1(t)]=cI^{\alpha}K_1(t).
\end{align*}
Let
\begin{align*}
K_1^0(t)&=a_1+AI^{\alpha}\left[E_\alpha(ct^\alpha)\right]^2,\\
K_1^1(t)&=M[K_1^0(t)]=\frac{a_1ct^{\alpha}}{\Gamma(\alpha+1)}+AcI^{2\alpha}\left[E_\alpha(ct^\alpha)\right]^2,\\
&\vdots\\
K_1^n(t)&=M[K_1^{n-1}(t)]=\frac{a_1c^nt^{n\alpha}}{\Gamma(\alpha+n)}+Ac^nI^{(n+1)\alpha}\left[E_\alpha(ct^\alpha)\right]^2.
\end{align*}
Hence \begin{equation}\label{2.2EX2.5}
K_1(t)=\sum_{m=0}^{\infty}K_1^m(t)=a_1E_\alpha(ct^\alpha)+A\sum_{m=0}^{\infty}c^mI^{(m+1)\alpha}\left[E_\alpha(ct^\alpha)\right]^2.
\end{equation}
Using (\ref{2.2EX2.4}) and (\ref{2.2EX2.5}), solution (\ref{2.2SOL2}) takes the following form:
\begin{align*}
f(t,x,y)=a_1E_\alpha(ct^\alpha)+A\sum_{m=0}^{\infty}c^mI^{(m+1)\alpha}\left[E_\alpha(ct^\alpha)\right]^2+\left[a_2E_\alpha(ct^\alpha)\right]x^\beta+\left[a_3E_\alpha(ct^\alpha)\right]y^{\gamma}.
\end{align*}
\subsubsection{Exact solution of fractional scale wave equation in (1+2) dimension}
	Consider fractional version of scale wave equation
	\begin{equation}\label{2.2EX3}
\frac{\partial ^{\beta}}{\partial x^{\beta}}\left(\frac{\partial ^{\beta}f}{\partial x^{\beta}}\right)+\frac{\partial ^{\gamma}}{\partial y^{\gamma}}\left(\frac{\partial ^{\gamma}f}{\partial y^{\gamma}}\right)-a\frac{\partial ^{\alpha}f}{\partial t^{\alpha}}-\frac{\partial ^{\alpha+1}f}{\partial t^{\alpha+1}}=0,~t>0,~\alpha,~\beta,~\gamma \in (0,1].
	\end{equation}
\noindent Comparing with the equation (\ref{2.HFPDE}) we note that
\begin{equation*}
N[f]=\frac{\partial ^{\beta}}{\partial x^{\beta}}\left(\frac{\partial ^{\beta}f}{\partial x^{\beta}}\right)+\frac{\partial ^{\gamma}}{\partial y^{\gamma}}\left(\frac{\partial ^{\gamma}f}{\partial y^{\gamma}}\right).
\end{equation*}
Observe that $I^2=\mathfrak{L}\{E_{\beta}(\lambda_1x^{\beta}),E_{\gamma}(-\lambda_2y^{\gamma})\}$ is one of the required invariant subspaces, since
\begin{align*}
N[k_1E_{\beta}(\lambda_1x^{\beta})+k_2E_{\gamma}(-\lambda_2y^{\gamma})]=\lambda_1^2k_1E_{\beta}(\lambda_1x^{\beta})+\lambda_2^2k_2E_{\gamma}(-\lambda_2y^{\gamma})\in I^2.
\end{align*}
Hence $f(t,x,y)=K_1(t)E_{\beta}(\lambda_1x^{\beta})+K_2(t)E_{\gamma}(-\lambda_2y^{\gamma})$ is a solution of the fractional scale wave equation (\ref{2.2EX3}), where $K_1(t)$ and $K_2(t)$ are the unknown functions that satisfy following system of FODEs.
\begin{align}
a\dfrac{ d ^{\alpha} K_1(t)}{ d t^ {\alpha}}+\dfrac{ d ^{\alpha+1} K_1(t)}{ d t^ {\alpha+1}}&=\lambda_1^2K_1(t),\label{2.2EX3.1}\\
a\dfrac{ d ^{\alpha} K_2(t)}{ d t^ {\alpha}}+\dfrac{ d ^{\alpha+1} K_2(t)}{ d t^ {\alpha+1}}&=\lambda_2^2 K_2(t),\label{2.2EX3.2}
\end{align}
Using Laplace transform technique to Eqs. (\ref{2.2EX3.1})-(\ref{2.2EX3.2}), we deduce the exact solution of fractional scale wave equation (\ref{2.2EX3}) as
\begin{align*}
f(t,x,y)&=\left[b_1\sum_{m=0}^{\infty}\frac{\lambda_2^{2m}}{m!}t^{(\alpha+1)m}E_{1,{\alpha}m+1}^{(m)}(-at)+\left(b_1+b_2\right)\sum_{m=0}^{\infty}\frac{\lambda_2^{2m}}{m!}t^{(\alpha+1)m+1}E_{1,{\alpha}m+2}^{(m)}(-at)\right]E_{\beta}(\lambda_2x^{\beta})\\
&+\left[a_1\sum_{m=0}^{\infty}\frac{\lambda_1^{2m}}{m!}t^{(\alpha+1)m}E_{1,{\alpha}m+1}^{(m)}(-at)+\left(a_1+a_2\right)\sum_{m=0}^{\infty}\frac{\lambda_2^{2m}}{m!}t^{(\alpha+1)m+1}E_{1,{\alpha}m+2}^{(m)}(-at)\right]E_{\gamma}(-\lambda_2y^{\gamma}),
\end{align*}where $a_1, a_2, b_1$ and $b_2 \in \mathbb{R}.$ 
\subsubsection{Solutions of fractional order Boussinesq equation}
	Consider the IVP for fractional order Boussinesq equation where $t>0,~\alpha,~ \beta \in (0,1].$ 
		\begin{align}
		&\frac{\partial ^{\alpha}f}{\partial t^{\alpha}}=\frac{\partial ^{\beta}}{\partial x^{\beta}}\left((rf+s)\frac{\partial ^{\beta}(rf+s)}{\partial x^{\beta}}\right)+\frac{\partial ^{\beta}}{\partial y^{\gamma}}\left((rf+s)\frac{\partial ^{\gamma}(rf+s)}{\partial y^{\gamma}}\right)=N[f],\label{2.2EX4}\\
	&	f(0,x,y)=\frac{9}{5}+e^2y^{\gamma},\label{2.2IC4}
		\end{align}where $r, s$ are arbitrary.
		For $\alpha=\beta=\gamma=1$, Eq. (\ref{2.2EX4}) is a two dimensional heat and mass transfer equation with temperature dependent diffusion coefficient \cite{zaitsev2003handbook}.\\
\noindent We choose $I^3=\mathfrak{L}\{1,x^\beta,y^\gamma\}$	which satisfies the invariant subspace property as follows:
\begin{equation*}
N[k_1+k_2x^\beta +k_3y^\gamma]=r^2\left[\Gamma(\beta+1)^2k_2^2+\Gamma(\gamma+1)^2k_3^2\right]\in I^3.
\end{equation*}	
In view of Theorem (\ref{2.THE2}), an exact solution of Eq. (\ref{2.2EX4}) has the form
\begin{equation*}\label{2.2SOL4}
f(t,x,y)=K_1(t)+K_2(t)x^\beta +K_3(t)y^\gamma,
\end{equation*}
where the unknown functions $K_1(t), K_2(t)$ and $K_3(t)$ are to be evaluated by solving the following system:
\begin{align}
\dfrac{ d ^{\alpha} K_1(t)}{ d t^ {\alpha}}&=r^2\left[\Gamma(\beta+1)^2K_2(t)^2+\Gamma(\gamma+1)^2K_3(t)^2\right],\label{2.2EX4.1}\\
\dfrac{ d ^{\alpha} K_2(t)}{ d t^ {\alpha}}&=0,\label{2.2EX4.2}\\
\dfrac{ d ^{\alpha} K_3(t)}{ d t^ {\alpha}}&=0\label{2.2EX4.3}.
\end{align}
Clearly $K_2(t)=a_2$ and $K_3(t)=a_3$, where $a_1$, $a_2$ are constants. Eq. (\ref{2.2EX4.1}) becomes
\begin{equation}\label{2.2EX4.4}
\dfrac{ d ^{\alpha} K_1(t)}{ d t^{\alpha}}=r^2[a_2^2\Gamma(\beta+1)^2+a_3^2\Gamma(\gamma+1)^2].
\end{equation}Solving Eq. (\ref{2.2EX4.4}), we get $
K_1(t)=a_1+\dfrac{r^2[a_2^2\Gamma(\beta+1)^2+a_3^2\Gamma(\gamma+1)^2]t^\alpha}{\Gamma(\alpha+1)}.$ Hence exact solution of IVP (\ref{2.2EX4})-(\ref{2.2IC4}) is
\begin{equation}\label{2.2SOL41}
	f(t,x,y)=\dfrac{9}{5}+\dfrac{e^4r^2\Gamma(\gamma+1)^2}{\Gamma(\alpha+1)}t^\alpha+e^2y^\gamma.
	\end{equation}
	The solution (\ref{2.2SOL41}) is plotted in Fig. 5.
	\begin{figure}[H]
	\centering
	\includegraphics[scale=0.5]{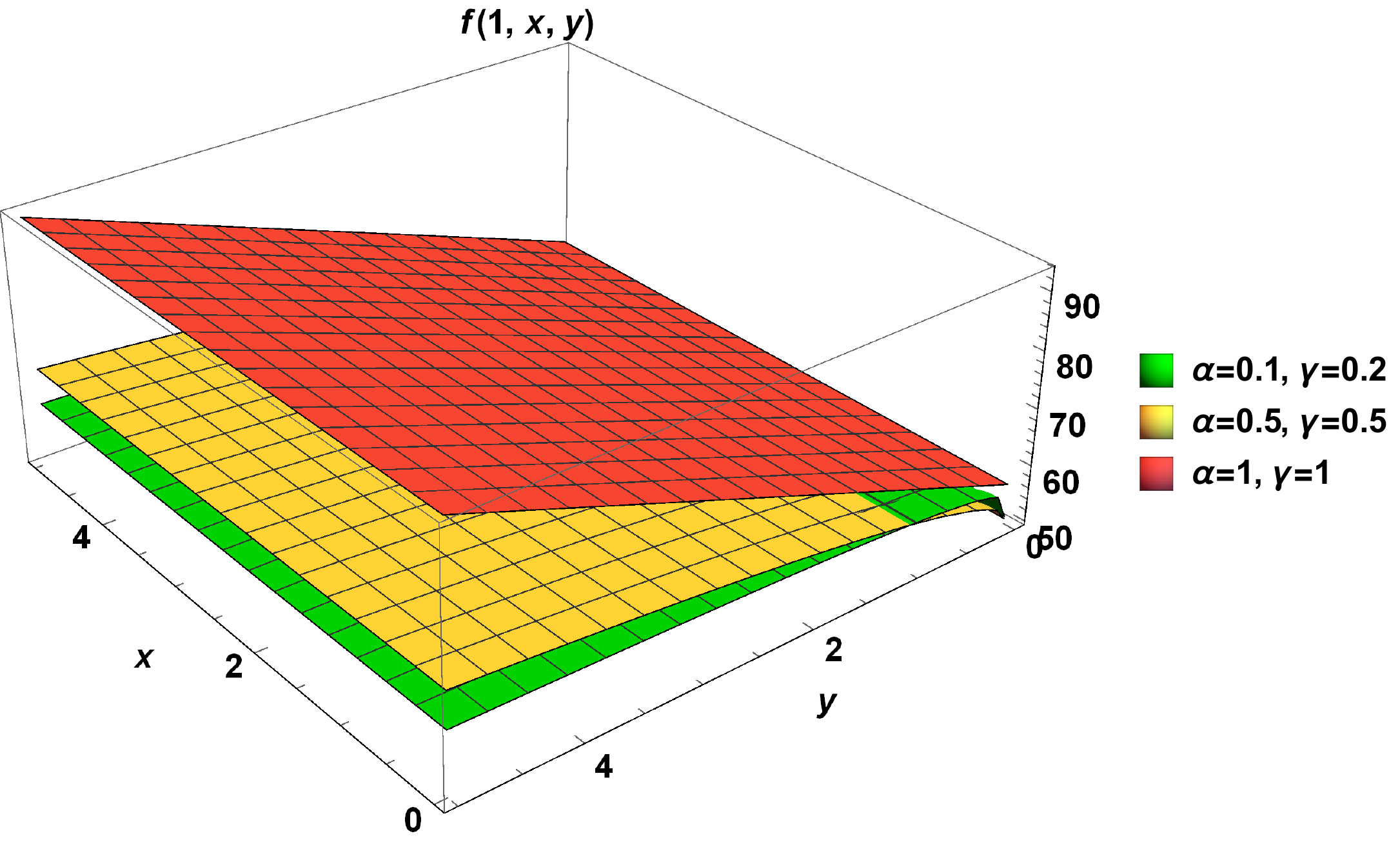}
	\caption{Plot of solution of the IVP (\ref{2.2EX4})-(\ref{2.2IC4}) for different values of $\alpha$ and $\gamma$ at $t=1.$}
\end{figure}	
\textbf{Note:} It can be verified that the Eq. (\ref{2.2EX4}) also admits another invariant subspace $I^4=\mathfrak{L}\{1,x^{2\beta},y^{2\gamma},x^\beta y^\gamma\}$. Thus proceeding on similar lines as previous examples another exact solution corresponding to $I^4$ can be found.
\subsubsection{Exact solution of fractional diffusion like PDE in (1+2) dimensions}
	Consider fractional PDE with initial condition as follows
	\begin{align}
	&\frac{\partial ^{\alpha}f}{\partial t^{\alpha}}=\frac{1}{2}\left(x^{2\beta}\frac{\partial ^{2\gamma}f}{\partial y^{2\gamma}}+y^{2\gamma}\frac{\partial ^{2\beta}f}{\partial x^{2\beta}}\right),~t>0,~\alpha,~\beta, ~\gamma \in (0,1],\label{2.2EX5}\\
&f(0,x,y)=y^{2\gamma}.\label{2.2IC5}
	\end{align}
\noindent	Here $I^3=\mathfrak{L}\{1,x^{2\beta},y^{2\gamma}\}$	is an invariant subspace for $N[f]=\dfrac{1}{2}\left(x^{2\beta}\dfrac{\partial ^{2\gamma}f}{\partial y^{2\gamma}}+y^{2\gamma}\dfrac{\partial ^{2\beta}f}{\partial x^{2\beta}}\right)$ as
\begin{equation*}
N[k_1+k_2x^{2\beta} +k_3y^{2\gamma}]=\frac{\Gamma(2\gamma+1)}{2}k_3x^{2\beta}+\frac{\Gamma(2\beta+1)}{2}k_2y^{2\gamma}\in I^3.
\end{equation*}	
Since criteria of Theorem (\ref{2.THE2}) is satisfied, an exact solution of Eq. (\ref{2.2EX5}) is of the form
\begin{equation}\label{2.2SOL5}
f(t,x,y)=K_1(t)+K_2(t)x^{2\beta} +K_3(t)y^{2\gamma},
\end{equation}
where the unknown functions $K_1(t), K_2(t)$ and $K_3(t)$ satisfy following system of FODEs:
\begin{align}
\dfrac{ d ^{\alpha} K_1(t)}{ d t^ {\alpha}}&=0,\label{2.2EX5.1}\\
\dfrac{ d ^{\alpha} K_2(t)}{ d t^ {\alpha}}&=\frac{\Gamma(2\gamma+1)}{2}K_3(t),\label{2.2EX5.2}\\
\dfrac{ d ^{\alpha} K_3(t)}{ d t^ {\alpha}}&=\frac{\Gamma(2\beta+1)}{2}K_2(t)\label{2.2EX5.3}.
\end{align}
Clearly from (\ref{2.2EX5.2})-(\ref{2.2EX5.3}), we deduce
\begin{equation}\label{2.2EX5.4}
\dfrac{ d ^{\alpha}}{ d t^ {\alpha}}\left(\dfrac{ d ^{\alpha} K_2(t)}{ d t^ {\alpha}}\right)=\lambda K_2(t),~~~\lambda=\lambda_1 \lambda_2,~\lambda_1=\frac{\Gamma(2\gamma+1)}{2},~\lambda_2=\frac{\Gamma(2\beta+1)}{2}.
\end{equation}
Taking Laplace transform of Eq. (\ref{2.2EX5.4}), we get
\begin{align*}
&s^{\alpha}\mathcal{L}(D^{\alpha}_tK_2(t);s)-s^{\alpha-1}\left(D^{\alpha}_tK_2(0)\right)=\lambda \tilde{K}_1(s)\\
&\tilde{K}_2(s)=b_1\frac{s^{2\alpha -1}}{s^{2\alpha}-\lambda}+b_2\frac{s^{\alpha -1}}{s^{2\alpha}-\lambda},~\text{where}~b_1=K_2(0),~b_2=D^{\alpha}K_2(0).
	\end{align*}
Performing inverse Laplace transform, we obtain $
K_2(t)=b_1 E_{2\alpha,1}(\lambda t^{2\alpha})+b_2 t^{\alpha}E_{2\alpha,\alpha+1}(\lambda t^{2\alpha}),$ where $b_1$ and $b_2$ are arbitrary. Substituting value of $K_2(t)$ in Eq. (\ref{2.2EX5.3}) and solving the same, we get $
K_3(t)=c+ \lambda_2 b_1t^{\alpha}E_{2\alpha,\alpha+1}(\lambda_2 t^{2\alpha})+\lambda_2 b_2 t^{2\alpha}E_{2\alpha,2\alpha+1}(\lambda_2 t^{2\alpha}).$\\
Hence we get an exact solution of Eq. (\ref{2.2EX5}) as
\begin{align}\label{2.2SOL5.1}
f(t,x,y)&=a+\left[b_1 E_{2\alpha,1}(\lambda t^{2\alpha})+b_2 t^{\alpha}E_{2\alpha,\alpha+1}(\lambda t^{2\alpha})\right]x^{2\beta}+\left[c+ \lambda_2 b_1t^{\alpha}E_{2\alpha,\alpha+1}(\lambda_2 t^{2\alpha})\right.\nonumber\\
&\left.+\lambda_2 b_2 t^{2\alpha}E_{2\alpha,2\alpha+1}(\lambda_2 t^{2\alpha})\right]y^{2\gamma},~~~c=b_2,~\text{and}~ a,b_1,b_2 \in \mathbb{R}.
\end{align}Using initial condition (\ref{2.2IC5}), solution (\ref{2.2SOL5.1}) reduces to
\begin{align}\label{2.2SOL5.2}
f(t,x,y)&=\left[ t^{\alpha}E_{2\alpha,\alpha+1}(\lambda t^{2\alpha})\right]x^{2\beta}+\left[1+\lambda_2  t^{2\alpha}E_{2\alpha,2\alpha+1}(\lambda_2 t^{2\alpha})\right]y^{2\gamma}.
\end{align}
\textbf{Note:} For $\alpha = \beta = \gamma =1,~c=b_2$ and  under an initial condition $f(0,x,y)=y^2$, solution (\ref{2.2SOL5.1}) reduces to 
\begin{align}\label{2.2SOL5.3}
f(t,x,y)=(\sinh t)x^2+(\cosh t)y^2.
\end{align}
Solution (\ref{2.2SOL5.3}) coincides with the solution for the heat like equation $\frac{\partial f}{\partial t}=\frac{1}{2}\left(x^2\frac{\partial ^{2}f}{\partial y^{2}}+y^{2}\frac{\partial ^{2}f}{\partial x^{2}}\right)$, obtained by NIM \cite{al2017daftardar}. \\
The solution (\ref{2.2SOL5.2}) is depicted in Fig. 6.
	\begin{figure}[H]
			\centering
		\includegraphics[scale=0.5]{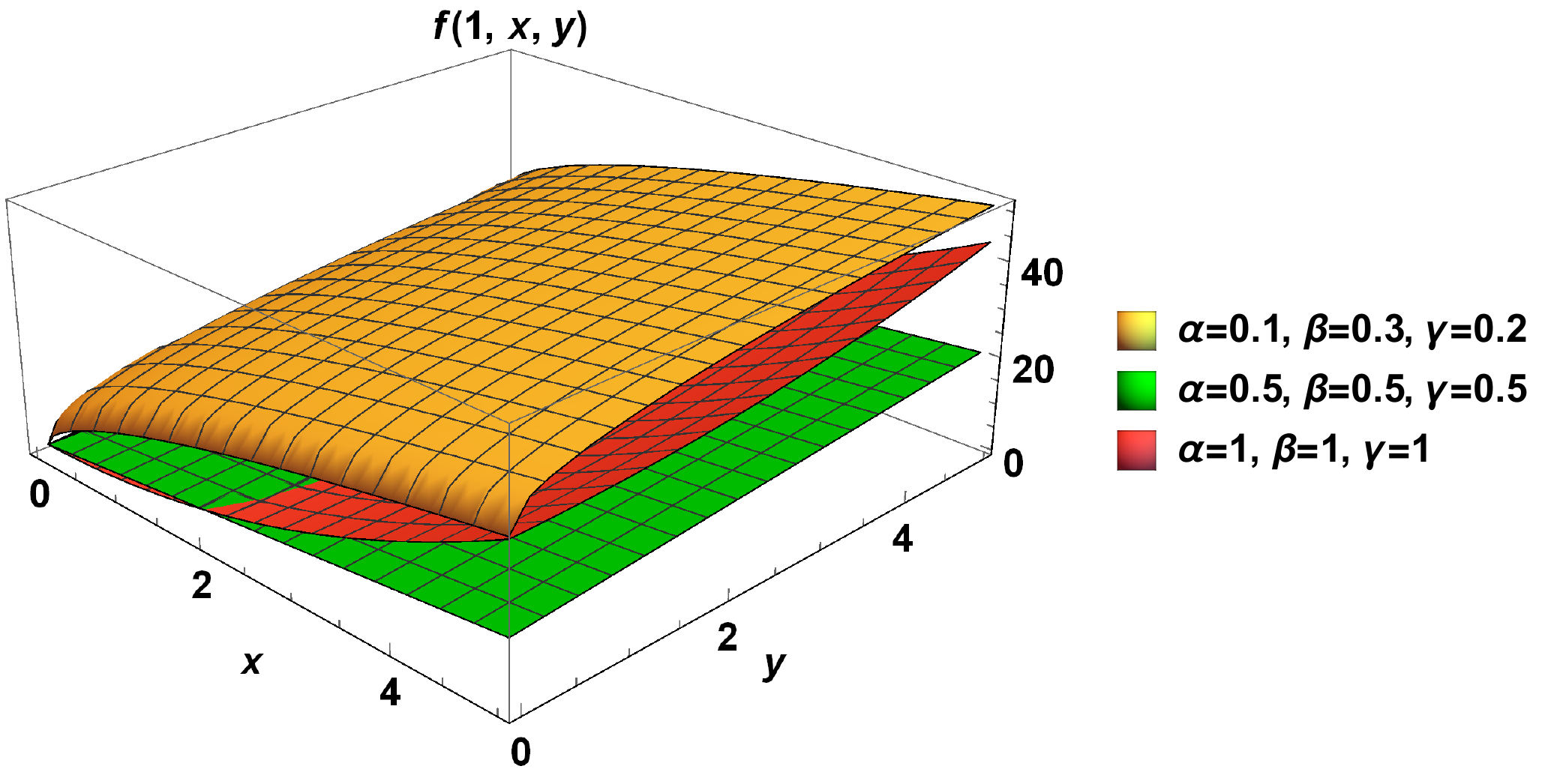}
		\caption{3D-Plot of solution of the IVP (\ref{2.2EX5})-(\ref{2.2IC5}) at $t=1.$}
\end{figure}
	
\section{Fractional differential operators with mixed partial derivatives}
It should be noted that invariant subspace method can also be employed for solving FPDEs with fractional differential operators $N^l[f]~(l=1,2)$ involving mixed fractional partial derivatives. Analysis of such FPDEs can be done on the similar lines as done in Sec. (3) and Sec. (4). We illustrate the method by solving an example.\\
 Consider the following system of nonlinear FPDEs for $t>0, 0< \alpha_1,~\alpha_2,~\beta,~\gamma\leq 1.$
\begin{align}\label{2.3EX1}
\frac{\partial^{\alpha_1}f}{\partial t^{\alpha_1}}&=\frac{\partial^{\gamma}}{\partial t^{\gamma}}\left(\frac{\partial^{2\beta}f}{\partial x^{2\beta}}\right)+m_1\left(g\frac{\partial^{\beta}g}{\partial x^{\beta}}\right)+a_1m_1g^2,\nonumber \\
\frac{\partial^{\alpha_2}g}{\partial t^{\alpha_2}}&=\frac{\partial^{\gamma}}{\partial t^{\gamma}}\left(\frac{\partial^{2\beta}g}{\partial x^{2\beta}}\right)+ n_1\frac{\partial^{\beta}}{\partial x^{\beta}}\left(\frac{\partial^{\beta}f}{\partial x^{\beta}}\right)-a_2^2n_1f+n_2g,
\end{align}	where $\gamma < \alpha_1, \gamma < \alpha_2,$ and $a_1,a_2,m_1,m_2,n_1$ and $n_2$ are arbitrary constants.\\

\noindent Observe that $I=I_1^2\times I_2^1=\mathfrak{L}\{E_{\beta}(a_2x^{\beta}),E_{\beta}(-a_2x^{\beta})\}\times \mathfrak{L}\{E_{\beta}(-a_1x^{\beta})\}$,\\ is an invariant subspace corresponding to the given fractional differential operator. Hence the system (\ref{2.3EX1}) admits solution of the form
\begin{equation}\label{2.3SOL1}
f(t,x)=K_1(t)E_{\beta}(a_2x^{\beta})+K_2E_{\beta}(-a_2x^{\beta}), ~g(t,x)=L_1(t)E_{\beta}(-a_1x^{\beta}),
\end{equation}such that
\begin{align}
\dfrac{ d ^{\alpha_1} K_1(t)}{ d t^ {\alpha_1}}&=a_2^2\dfrac{ d ^{\gamma} K_1(t)}{ d t^ {\gamma}},\label{2.3EX1.2}\\
\dfrac{ d ^{\alpha_1} K_2(t)}{ d t^ {\alpha_1}}&=a_2^2\dfrac{ d ^{\gamma} K_2(t)}{ d t^ {\gamma}},\label{2.3EX1.3}\\
\dfrac{ d ^{\alpha_2} L_1(t)}{ d t^ {\alpha_2}}&=a_1^2\dfrac{ d ^{\gamma} L_1(t)}{ d t^ {\gamma}}+n_2L_1(t).\label{2.3EX1.4}
\end{align}
Solving the system of FODEs (\ref{2.3EX1.2})-(\ref{2.3EX1.4}) and substituting the values of $K_1(t), K_2(t)$ and $L_1(t)$ in Eq. (\ref{2.3SOL1}) we get an exact solution of the system (\ref{2.3EX1}) as
\begin{align*}
f(t,x)&=\left[b_1\sum_{m=0}^{\infty}\left(\dfrac{a_2^{2m}t^{(\alpha_1-\gamma)m}}{\Gamma((\alpha_1-\gamma)m+1)}-\dfrac{a_2^{2m+2}t^{(\alpha_1-\gamma)(m+1)}}{\Gamma((\alpha_1-\gamma)(m+1)+1)}\right)\right]E_{\beta}(a_2x^{\beta})\\
&+\left[b_2\sum_{m=0}^{\infty}\left(\dfrac{a_2^{2m}t^{(\alpha_1-\gamma)m}}{\Gamma((\alpha_1-\gamma)m+1)}-\dfrac{a_2^{2m+2}t^{(\alpha_1-\gamma)(m+1)}}{\Gamma((\alpha_1-\gamma)(m+1)+1)}\right)\right]E_{\beta}(-a_2x^{\beta}),\\
g(t,x)&=\left[c_1\sum_{m=0}^{\infty}\left(\frac{n_2^m}{m!}t^{\alpha_2 m}E_{\alpha_2-\gamma, \gamma m+1}^{(m)}(a_1^2t^{\alpha_2-\gamma})-\frac{a_1^2n_2^m}{m!}t^{2\alpha_2 -2\gamma}E_{\alpha_2-\gamma,\alpha_2+\gamma( m-1)+1}^{(m)}(a_1^2t^{\alpha_2-\gamma})\right)\right]E_{\beta}(-a_1x^{\beta}),
\end{align*}where $b_1,b_2,c_1$ are arbitrary constants.
	\section{Conclusions and future scope}
		Present article extends invariant subspace method for solving nonlinear systems of FPDEs involving both time and space fractional derivatives. In this method system of FPDEs are reduced to systems of FODEs which can be further solved by existing methods. The proposed method has been illustrated by finding exact solutions of various systems, \textit{viz.,} system of generalized fractional Burger's equations, coupled fractional Boussinesq equations, fictionalized system of KdV type of equations. Further we demonstrate how invariant subspace method can be employed for FPDEs in (1+n) dimension. The effectiveness of this method is illustrated by finding closed form solutions for                                                                                                                                                                                                                                                                                                                                                                                                                                                                                                                                                                          fractional dispersive KdV equation in (1+n) dimensions, fractional population model, fractional scale wave equation, fractional order Bossinesq equation and fractional diffusion like PDE in (1+2) dimensions We have modelled equations using RL as well as Caputo derivatives and considered multi-term expressions in time. Invariant subspace method is also used to find unique solutions along with initial conditions.
		
		We observe that (1+1) dimensional FPDEs admit more than one invariant subspaces, each of which yields different exact solution \cite{choudhary2017invariant}. Similarly FPDEs in higher dimensions admit more than one invariant subspaces. The solutions obtained can be expressed in terms of Mittag-Leffler functions, fractional trigonometric functions etc. We demonstrate that invariant subspace method is very effective tool in finding exact solutions of wide class of linear and non linear systems of FPDEs and FPDEs in higher dimensions. Further we have also employed invariant subspace method for solving FPDEs with fractional differential operator involving mixed fractional partial derivatives.
		
	Due to lack of composition rule and chain rule, we have severe limitations in finding invariant subspaces corresponding to fractional operator using existing algorithms developed for ordinary PDEs. We have found invariant subspaces for the fractional operators by trial and error method. Developing proper theory and algorithms for finding all sets of  invariant subspaces for fractional operators is an open area to explore. Similarly, suitable theory may be developed for finding maximum dimension of invariant subspaces.  

	\vskip 0.3cm

\section*{Acknowledgement}

Sangita Choudhary acknowledges the National Board for Higher Mathematics, India, for the award of Senior Research Fellowship.
	\bibliography{4BIB.bib}

\begin{thebibliography}{10}

\bibitem{adomian1989solution}
G.~Adomian, R.~Rach, and M.~Elrod.
\newblock On the solution of partial differential equations with specified
  boundary conditions.
\newblock {\em Journal of Mathematical Analysis and Applications},
  140(2):569--581, 1989.

\bibitem{al2017daftardar}
W.~Al-Hayani.
\newblock Daftardar-jafari method for fractional heat-like and wave-like
  equations with variable coefficients.
\newblock {\em Applied Mathematics}, 8(02):215, 2017.

\bibitem{bhalekar2008new}
S.~Bhalekar and V.~Daftardar-Gejji.
\newblock New iterative method: application to partial differential equations.
\newblock {\em Applied Mathematics and Computation}, 203(2):778--783, 2008.

\bibitem{bhalekar2011fractional}
S.~Bhalekar, V.~Daftardar-Gejji, D.~Baleanu, and R.~Magin.
\newblock Fractional bloch equation with delay.
\newblock {\em Computers \& Mathematics with Applications}, 61(5):1355--1365,
  2011.

\bibitem{bonilla2007fractional}
B.~Bonilla, M.~Rivero, L.~Rodr{\'\i}guez-Germ{\'a}, and J.~J. Trujillo.
\newblock Fractional differential equations as alternative models to nonlinear
  differential equations.
\newblock {\em Applied Mathematics and Computation}, 187(1):79--88, 2007.

\bibitem{choudhary2017invariant}
S.~Choudhary and V.~Daftardar-Gejji.
\newblock Invariant subspace method: a tool for solving fractional partial
  differential equations.
\newblock {\em Fractional Calculus and Applied Analysis}, 20(2):477--493, 2017.

\bibitem{debnath2003recent}
L.~Debnath.
\newblock Recent applications of fractional calculus to science and
  engineering.
\newblock {\em International Journal of Mathematics and Mathematical Sciences},
  2003(54):3413--3442, 2003.

\bibitem{diethelm2010analysis}
K.~Diethelm.
\newblock {\em The analysis of fractional differential equations: An
  application-oriented exposition using differential operators of Caputo type}.
\newblock Springer, 2010.

\bibitem{galaktionov2006exact}
V.~A. Galaktionov and S.~R. Svirshchevskii.
\newblock {\em Exact solutions and invariant subspaces of nonlinear partial
  differential equations in mechanics and physics}.
\newblock CRC Press, 2006.

\bibitem{gazizov2013construction}
R.~K. Gazizov and A.~A. Kasatkin.
\newblock Construction of exact solutions for fractional order differential
  equations by the invariant subspace method.
\newblock {\em Computers \& Mathematics with Applications}, 66(5):576--584,
  2013.

\bibitem{gurtin1977diffusion}
M.~E. Gurtin and R.~C. MacCamy.
\newblock On the diffusion of biological populations.
\newblock {\em Mathematical Biosciences}, 33(1-2):35--49, 1977.

\bibitem{hirota1981soliton}
R.~Hirota and J.~Satsuma.
\newblock Soliton solutions of a coupled Korteweg-de Vries equation.
\newblock {\em Physics Letters A}, 85(8-9):407--408, 1981.

\bibitem{inan2007exact}
I.~E. Inan and D.~Kaya.
\newblock Exact solutions of some nonlinear partial differential equations.
\newblock {\em Physica A: Statistical Mechanics and its Applications},
  381:104--115, 2007.

\bibitem{ito1982symmetries}
M.~Ito.
\newblock Symmetries and conservation laws of a coupled nonlinear wave
  equation.
\newblock {\em Physics Letters A}, 91(7):335--338, 1982.

\bibitem{jafari2013new}
H.~Jafari, M.~Nazari, D.~Baleanu, and C.~M. Khalique.
\newblock A new approach for solving a system of fractional partial
  differential equations.
\newblock {\em Computers \& Mathematics with Applications}, 66(5):838--843,
  2013.

\bibitem{miller1993introduction}
K.~S. Miller and B.~Ross.
\newblock An introduction to the fractional calculus and fractional
  differential equations.
\newblock 1993.

\bibitem{ovsiannikov2014group}
L.~V. Ovsiannikov.
\newblock {\em Group analysis of differential equations}.
\newblock Academic Press, 2014.

\bibitem{podlubny1998fractional}
I.~Podlubny.
\newblock {\em Fractional differential equations: an introduction to fractional
  derivatives, fractional differential equations, to methods of their solution
  and some of their applications}, volume 198.
\newblock Elsevier, 1998.

\bibitem{qu2009classification}
C.~Qu and C.~Zhu.
\newblock Classification of coupled systems with two-component nonlinear
  diffusion equations by the invariant subspace method.
\newblock {\em Journal of Physics A: Mathematical and Theoretical},
  42(47):475201, 2009.

\bibitem{sahadevan2016exact}
R.~Sahadevan and P.~Prakash.
\newblock Exact solution of certain time fractional nonlinear partial
  differential equations.
\newblock {\em Nonlinear Dynamics}, 85(1):659--673, 2016.

\bibitem{sahadevan2017exact}
R.~Sahadevan and P.~Prakash.
\newblock Exact solutions and maximal dimension of invariant subspaces of time
  fractional coupled nonlinear partial differential equations.
\newblock {\em Communications in Nonlinear Science and Numerical Simulation},
  42:158--177, 2017.

\bibitem{sahadevan2017lie}
R.~Sahadevan and P.~Prakash.
\newblock On lie symmetry analysis and invariant subspace methods of coupled
  time fractional partial differential equations.
\newblock {\em Chaos, Solitons \& Fractals}, 104:107--120, 2017.

\bibitem{srivastava2014one}
V.~K. Srivastava, M.~Tamsir, M.~K. Awasthi, and S.~Singh.
\newblock One-dimensional coupled burgers’ equation and its numerical
  solution by an implicit logarithmic finite-difference method.
\newblock {\em Aip Advances}, 4(3):037119, 2014.

\bibitem{wu2018method}
C.~Wu and W.~Rui.
\newblock Method of separation variables combined with homogenous balanced
  principle for searching exact solutions of nonlinear time-fractional
  biological population model.
\newblock {\em Communications in Nonlinear Science and Numerical Simulation},
  2018.

\bibitem{zaitsev2003handbook}
V.~F. Zaitsev and A.~D. Polyanin.
\newblock Handbook of nonlinear partial differential equations., 2003.

\end{thebibliography}
\bibliographystyle{abbrv}

\end{document}